\documentclass[11pt,oneside]{amsart}
\usepackage[english]{babel}
\usepackage{amssymb,enumerate,latexsym,hyperref}
\usepackage[monochrome]{xcolor}
\usepackage{marginnote}
\usepackage{color}
\usepackage{geometry}
\newcommand{\Tau}{\mathrm{T}}
\newcommand{\assign}{:=}

\newcommand{\red}[1]{\textcolor{red}{#1}} % 在导言区定义
\newcommand{\tmtextbf}[1]{\text{{\bfseries{#1}}}}

\newenvironment{proof*}[1]{\medskip\noindent\textbf{#1\ }}{\hspace*{\fill}$\Box$\medskip}
\newtheorem{theorem}{Theorem}[section]
\newtheorem{lemma}[theorem]{Lemma}
\newtheorem{proposition}[theorem]{Proposition}
\newtheorem{corollary}[theorem]{Corollary}
\newtheorem{definition}[theorem]{Definition}
\theoremstyle{remark}
\newtheorem{remark}[theorem]{Remark}
\newtheorem{example}[theorem]{Example}
\newtheorem{problem}[theorem]{Problem}
\newtheorem{conjecture}{\bfseries{Conjecture}}

\begin{document}

\title[holomorphic partially hyperbolic systems]{On Holomorphic partially hyperbolic systems} 

\author{disheng xu}
\address{School of Science, Great Bay University and Great bay institute for advanced study, 
Songshan Lake International Innovation Entrepreneurship Community A5, Dongguan 523000, CHINA}
\email{xudisheng@gbu.edu.cn}

\author{jiesong zhang}
\address{School of Mathematical Sciences, Peking University, No.5 Yiheyuan Road, Haidian District, Beijing 100871, China}
\email{zhjs@stu.pku.edu.cn}

\begin{abstract}
We prove that the center distribution of any fibered holomorphic partially hyperbolic diffeomorphism on a complex $3$-fold is holomorphic. In particular, any such a system is a holomorphic skew product over a linear automorphism on a complex $2$-torus. In higher dimension, we demonstrate a contrast: for any $n\geq 5$, we construct a holomorphic fibered partially hyperbolic system on a complex $n$-fold, where the center distribution is not holomorphic in any open subset.
%We construct examples illustrating that dynamically-defined distributions of holomorphic diffeomorphisms on compact complex manifolds are not necessarily holomorphic in any open subset. More precisely, for any $n\geq 5$, we construct a holomorphic fibered partially hyperbolic system on a complex $n$-fold, where the center distribution is not holomorphic in any open subset. For $n=3$ we demonstrate a contrast: the center distribution of any fibered holomorphic partially hyperbolic diffeomorphism 
%on a complex $3$-fold is holomorphic. In particular, any such a system is a holomorphic skew product over a linear automorphism on a complex $2$-torus.
\end{abstract}

{\maketitle}

{\tableofcontents}

\section{Introduction}
\subsection{A brief introduction to Anosov systems, partially hyperbolic systems and pathological foliations}
Dynamical systems exhibiting hyperbolic behavior constitute a broad and fundamental topic in the theory of dynamical systems, the strongest form of this behavior is known as uniform hyperbolicity, also referred to as Anosov systems. Recall that a diffeomorphism $f : M \rightarrow M$ on a compact Riemannian manifold $M$ is called \textit{Anosov} if the tangent bundle
$TM$ allows a $Df$-invariant splitting $E^s \oplus E^u$ and there exists a positive integer $k$ such that for any $x\in M$, $$\|Df^k|_{E^s(x)}\|<1<\|(Df^k|_{E^u(x)})^{-1}\|^{-1}.$$
All known examples of Anosov diffeomorphisms are \textit{topologically} conjugate to automorphisms of infra-nilmanifolds. An open conjecture that dates back to Anosov and Smale is that there is no other example of Anosov diffeomorphisms \cite{sma67}. Franks \cite{fra70} and Newhouse \cite{new70} confirmed the conjecture in codimension-$1$ case, but the higher codimension case remains widely open. A remarkable aspect is that $E^{s,u}$ are always tangent to H\"older continuous foliations $W^{s,u}$ respectively, which have $C^1$ leaves and are \textit{absolutely continuous} \cite{bs02}. 

Partially hyperbolic systems are natural generalizations of Anosov systems.
A diffeomorphism $f : M \rightarrow M$ is called partially hyperbolic if there exists a continuous $Df$-invariant splitting of $TM$, 
$TM=E^s \oplus E^c \oplus E^u$ such that for any $x\in M$, 
\begin{eqnarray*}
    &&\|Df^k|_{E^s(x)}\|<1<\||(Df^k|_{E^u(x)})^{-1}\|^{-1};\\
    &&\|Df^k|_{E^s(x)}\|<\|(Df^k|_{E^c(x)})^{-1}\|^{-1},\quad \|Df^k|_{E^c(x)}\|<\|(Df^k|_{E^u(x)})^{-1}\|^{-1}.
\end{eqnarray*}
In particular a partially hyperbolic diffeomorphism is Anosov if $E^c$ is trivial. Examples of partially hyperbolic systems include the time $t$-map of Anosov flows, suspension of Anosov systems, skew product over Anosov systems, etc. 

Partially hyperbolic systems with non-trivial center distribution $E^c$ demonstrate more diverse dynamical phenomena compared to Anosov systems. 
Similar to Anosov systems, $E^{s,u}$ of partially hyperbolic systems are integrable,  and the integral foliations, denoted again by $W^{s,u}$, are absolutely continuous as that of Anosov systems. However, the center distribution $E^c$ could be non-integrable \cite{wil98,hhu16}. Even the bundle $E^c$ is tangent to a compact foliation $W^c$, in general, $W^c$ can still be non-absolutely continuous. To illustrate this we consider an important class of partially hyperbolic system.
\begin{definition}
Let $f$ be a partially hyperbolic diffeomorphism of a closed manifold $M$, $f$ is called 
\textit{fibered} if $E^c$ is tangent to a compact foliation $W^c$ and $W^c$ is a topological fibration of $M$, i.e. the quotient space $\hat M = M/ W^c$ is a topological manifold. And the map $\hat f :\hat M \to \hat M$ is called the \textit{base map}. 

If in addition that the manifold $M$ is a complex manifold, $f$ is holomorphic and 
$W^c$ is a holomorphic foliation (see Section \ref{journee}), then $f$ is called a \textit{holomorphic skew product}. In this case, $\hat M$ is also a compact complex manifold and $\hat f$ is a holomorphic Anosov diffeomorphism. 
\end{definition}

For partially hyperbolic systems with compact center leaves, the \textit{fibered} condition is not quite restrictive. It is expected that if all center leaves of a partially hyperbolic diffeomorphism $f$ are compact then there exists a finite cover of $f$  is fibered (see \cite{boh11} and references therein for some partial results about this problem). 

In general, the center foliations of fibered partially hyperbolic systems may not be absolutely continuous. Shub and Wilkinson \cite{sw00} (see also Ruelle-Wilkinson \cite{rw01}) constructed a real analytic stably ergodic fibered partially hyperbolic system on $\mathbb T^3$ such that there exists a full volume set that intersects every $W^c$-leaf at only a finite number of points. This phenomenon, commonly known as ``Fubini's nightmare", results in such 
$W^c$ being labeled as \textit{pathological}. 
A non-ergodic example of pathological center foliation is given by Katok \cite{mil97} by different mechanism, see also Section \ref{Subsec: sketch/plan}.

\subsection{Holomorphic diffeomorphisms and \textit{holomorphic rigidity}}\label{subsec: Ghys}
Unlike the dynamics induced by smooth diffeomorphisms, the dynamics of holomorphic diffeomorphisms are expected to exhibit stronger rigidity properties, especially in lower dimensions.
For instance, if a holomorphic diffeomorphism on a compact Kähler manifold preserves a Kähler form or is homotopic to the identity, then it exhibits simple dynamics characterized by zero entropy \cite{gro03}.
In the case that with positive entropy, E. Ghys \cite{ghy95} established a result on holomorphic rigidity, proved that any holomorphic Anosov diffeomorphism on a complex surface is a linear \footnote{A map $f:\mathbb C^n/\Lambda\to \mathbb C^n/\Lambda$ on is called linear if $f(x)=Ax+b (\mod \Lambda), A\in  GL(n,\mathbb C), A\cdot \Lambda=\Lambda, b\in \mathbb C^n$.} automorphism on a complex torus. Ghys also
suggested a conjecture that 
a similar holomorphic rigidity phenomenon might hold for higher-dimensional holomorphic Anosov systems
(Page 586, \cite{ghy95}):
\begin{conjecture}\label{conj: Ghys}
All holomorphic Anosov diffeomorphisms are holomorphically conjugate to automorphisms of complex infra-nilmanifolds.
\end{conjecture}
This conjecture remains open, even in dimension 3. A significant challenge in proving this conjecture lies in establishing the holomorphicity of $E^{u,s}$. In fact it is generally unknown that $E^{u,s}$ are $C^1$. 
Ghys \cite{ghy95} proved that $E^{u}$ (or $E^s$) is holomorphic if it is of codimension one.  
Assume that $E^{u,s}$ both are holomorphic, 
Cantat \cite{can04} proved Conjecture \ref{conj: Ghys} holds for projective varieties.
\subsection{Non-holomorphic dynamically-defined invariant distributions} In fact, holomorphic diffeomorphisms may be more complicated than one might initially expect, and they may not exhibit the same level of rigidity as demonstrated in the studies \cite{ghy95, can04}. In this paper, we construct an example to show that holomorphic diffeomorphisms, especially in higher-dimensional settings with positive entropy, can have dynamically-defined invariant distributions that fail to be holomorphic on any non-empty open subset. These findings imply that such systems cannot be holomorphically conjugated to any algebraically defined systems.
More precisely: 
\begin{theorem}\label{thm: constru 5d}
  For any $n \geqslant 5$, there exists a compact complex manifold $M^n$ of
  dimension $n$ and a fibered holomorphic partially hyperbolic diffeomorphism $f$
  on $M^n$ such that the center distribution is real analytic but not 
  holomorphic on any non-empty open subset (implying that the center fibration $W^c$ is not holomorphic).
\end{theorem}
Theorem \ref{thm: constru 5d} can be viewed as a holomorphic analogue of pathological center foliations. For further discussions, see Sections \ref{Subsec: sketch/plan} and \ref{sec non holo ex}.

\subsection{Low dimensional holomorphic rigidity and results on accessibility classes}
In the context of low-dimensional cases, we show a contrast to Theorem \ref{thm: constru 5d}.
\begin{theorem}
  \label{t2}
  Let $f$ be a holomorphic fibered partially hyperbolic diffeomorphism of a compact complex $3$-fold $M$. Then $E^{cs},E^{cu}$ and $E^c$ are holomorphic distributions. Moreover $f$ %Moreover if $f$ is fibered then $f$ 
  is a holomorphic skew product over a linear automorphism on a complex $2$-torus.
\end{theorem}
\begin{remark}
For the first claim in Theorem \ref{t2} regarding the regularity of $E^{c, cs,cu}$,
we can prove it under a slightly weaker assumption than fiberedness. Specifically, the center leaves of $f$ are \textit{uniformly compact}, which implies that $M/W^c$ is a Hausdorff space (refer to \cite{boh11} and references therein for discussions on compact center foliations of partially hyperbolic systems). This adaptation requires only  minor changes to the proof \footnote{Bohnet's result \cite{boh11} indicates that if all the center leaves of $f$ are uniformly compact, then by excluding finitely many center fibers, the manifold, while not compact, becomes fibered. The discussions in Section \ref{sec dic} and \ref{sec gib} remain valid. Considering that $C^0$-limits of $\mathbb R$-linear and holomorphic mappings retain their respective properties, the methodologies in Section \ref{sec contra}, \ref{sec con} and \ref{sec d1} are applicable to the uniformly compact case as well. This approach ensures that center holonomy is holomorphic on a dense $c$-saturated subset, then by taking limits we can extend holomorphicity of the center holonomy everywhere and completes the proof.} To avoid heavy notations and redundant repetition of similar arguments, we sketch these modifications  in the footnote for interested readers.

The necessity of $f$ being fibered for the second claim of Theorem \ref{t2} is showed in Section \ref{sec exa}, where we provide an example that the center leaves of $f$ are all compact without the holomorphic skew product structure.
\end{remark}
An intriguing characteristic of the systems described in Theorem \ref{t2} is that, despite $E^c$ being holomorphic, the dynamics of $f$ can remain quite complicated. For instance:
\begin{theorem}\label{thm: 5d acc class}There is a fibered holomorphic partially hyperbolic diffeomorphism $f$ on a complex 3-fold where every accessibility class of $f$
forms a 5-dimensional
(real)-submanifold, which consequently is not a complex submanifold.
\end{theorem}
For the definition of accessibility classes and more details about this example, see Section \ref{sec exa}.

\subsection{Sketch of the proof and plan of the article}\label{Subsec: sketch/plan} The novelty of our article is to connect regularity of dynamically-invariant distributions to \textbf{heat flows} in Riemannian geometry and \textbf{deformations} of complex tori in complex geometry. We provide a sketch of the proofs and the structure of our paper as follows.

After the Preliminaries (Section \ref{pre}), the paper is divided into two main parts. The first part addresses the proof of Theorem \ref{t2}, and the second part is devoted to the construction of examples that satisfy Theorem \ref{thm: constru 5d} and Theorem \ref{thm: 5d acc class}.
%The proof of Theorem \ref{t2} is structured in two segments. The first segment focuses on establishing the $C^\infty$property of the center distribution. 
Demonstrating regularity of dynamically invariant distributions could be challenging, and we tackle this by classifying the dynamics of 
$f$ on the center foliation into \textit{isometries} and \textit{contractions}. 
The proof for the contraction case is based on certain holomorphic rigidity properties and arguments in quasiconformal geometry (see Section \ref{sec contra}). The isometric case,  on the other hand, involves some new ideas and careful approaches:
\\
\paragraph{\textbf{Uniformly quasiconformal partially hyperbolic systems, heat flows, the existence of Gibbs $cu$-states and $C^\infty$ property of $E^c$.}}

The first step is to adapt the theory of quasiconformal geometry and the study of uniformly quasiconformal partially hyperbolic systems \cite{bx18} to our context and show $E^c$ is $C^\infty$. In \cite{bx18}, the first-named author and C. Butler demonstrated that, under volume-preserving conditions, if a fibered partially hyperbolic system $f$
is uniformly quasiconformal (see Section \ref{subsec qc}) and $\infty$-bunched (basically means the sub-exponential growth of $Df^n|_{E^c}$), then $E^c$ is 
$C^\infty$. However, direct applying of \cite{bx18} to our situation is not straightforward, as holomorphic diffeomorphisms do not necessarily preserve volume. We therefore expand the applicability of \cite{bx18} to include certain dissipative settings %replacing $\infty$-bunching condition by a slightly stronger center-isometric assumption 
(see Proposition \ref{s5} and Proposition \ref{reli}). 

A crucial element of our aprroach here is employing of a center-leafwise heat flow to establish the existence of Gibbs $cu/cs$-states. Unlike Gibbs-$u$ state, constructing Gibbs-$cu/cs$ states for partially hyperbolic systems lacks a universal method and they may not exist. This part may be interesting on its own and could lead to new uses in the study of partially hyperbolic systems. The detailed arguments for this approach are in Section \ref{sec con}, which are adaptable for more general settings.
\\
\paragraph{\textbf{Deformation of complex structures (I): moduli space of $\mathbb T^2$ and holomorphicity of $E^c$ for $3$d case.}}
Section \ref{sec d1} constitutes the second step of the proof of Theorem \ref{t2}. With the constraints of lower dimensions as outlined in Theorem \ref{t2}, we carefully construct translation maps in the center-unstable (center-stable) leaf via the dynamics. These dynamically defined maps have nice geometric properties and establish an unexpected linkage between the regularity of the center foliation and holomorphic maps from the center leaves to $\mathrm{Mod}(\mathbb{T}^2)$, the moduli space (space of complex structures) of $\mathbb T^2$. According to a theorem in complex geometry, such maps have to be constant, which forces the center foliation to be holomorphic.
\\
\paragraph{\textbf{Deformation of complex structures (II): Non-K\"ahler complex structures on higher dimensional tori and constructions of non-holomorphic center distributions}.}
In the realm of higher-dimensional cases, the procedure for constructing translation maps in the lower-dimensional case does not work. Even if the well-behaved translation maps do exist and we have a similar linkage of the regularity of the center foliation to maps from the center leaves to $\mathrm{Mod}(\mathbb{T}^{2n})$ $(n \geq 2)$, the moduli space is ``large"  enough to support certain non-trivial maps, which implies non-holomorphic center distributions in Theorem \ref{thm: constru 5d}. The detailed construction is in Section \ref{sec non holo ex}. We start with the construction of Blanchard-Calabi manifolds, i.e., a non-K\"ahler complex structure on $\mathbb{T}^{2n+2}$ $(n \geq 2)$, and then modify the construction to make it support holomorphic fibered partially hyperbolic systems.

{\medskip}
In Section \ref{sec exa} we show some interesting examples of partially hyperbolic holomorphic diffeomorphisms, including that satisfies Theorem \ref{thm: 5d acc class}.  Section \ref{sec exa} also discusses examples of accessible fibered holomorphic partially hyperbolic system on Iwasawa manifolds.

\subsection{Further problems}
Theorem \ref{t2} leads us to propose the following conjecture about holomorphic partially hyperbolic systems of dimension $3$. Recall that holomorphic Anosov flows were defined and partially classified in \cite{ghy95}. Notably, the time-$t$ map of a holomorphic Anosov flow is partially hyperbolic.\footnote{In \cite{ghy95}, Ghys defined holomorphic Anosov flows using a $\mathbb{C}^{\ast}$-action, so to get a genuinely partially hyperbolic diffeomorphism, $\|t\| \neq 1$.}
\begin{conjecture}\label{conj: class 3d ph}
All holomorphic partially hyperbolic diffeomorphisms on 3-manifolds (up to a finite cover) are holomorphically conjugate to one of the following:
\begin{itemize}
  \item Automorphisms on complex tori;
  \item Holomorphic skew products over linear Anosov automorphisms;
  \item Time-$t$ maps of holomorphic Anosov flows.
\end{itemize}
\end{conjecture}
Conjecture \ref{conj: class 3d ph} may be challenging since the corresponding conjecture for the Anosov case remains open: 
Ghys \cite{ghy95} proved that a transitive holomorphic Anosov diffeomorphism on a complex 3-fold is topologically conjugate to an automorphism of complex tori, but it remains unknown whether such conjugacy is holomorphic. In addition, as we mentioned, it is not easy to classify the second class in Conjecture \ref{conj: class 3d ph}, see also Section \ref{sec exa}.

From Theorem \ref{thm: constru 5d}, we believe that pathological center foliations are unlikely to exist in the holomorphic context, leading to the following conjecture:
\begin{conjecture}
The center distribution of a holomorphic partially hyperbolic diffeomorphism is real analytic. 
\end{conjecture}

In particular, considering that most of ``interesting" examples we construct in the paper are not on K\"ahler manifolds, it is natural to ask: 
\begin{conjecture}
The center distribution of a holomorphic partially hyperbolic diffeomorphism on a K\"ahler manifold is holomorphic. 
\end{conjecture}
And another natural problem follows Theorem \ref{t2} is:
\begin{problem}
Is there  a compact complex $4$-fold support holomorphic fibered partially hyperbolic system with non-holomorphic center distribution?  
\end{problem}

{\noindent}\tmtextbf{Acknowledgment. }\color{red}The authors would first and foremost like to thank the anonymous referee for the careful reading our initial draft.
\color{black}The authors would like to thank Yuxiang Jiao,
Junyi Xie and Ziyi Zhao for helpful discussions. D. X. would like
to thank Artur Avila and Serge Cantat for useful suggestions at the early
stage of this work. {\medskip}

\section{Preliminaries}\label{pre}

\subsection{Preliminaries in complex geometry}

This subsection contains some basic results and discussions in complex geometry for the convenience of readers (see for example \cite{voi02} for more details).

A \textit{complex manifold} is a manifold with a complex structure, i.e., an atlas of charts to open sets in $\mathbb{C}^n$ such that the transition maps are holomorphic. A \textit{Hermitian manifold} is the complex analogue of a Riemannian manifold. More precisely, a Hermitian manifold is a complex manifold $M$ with a $C^\infty$-ly varying Hermitian inner product $h$ on the tangent space at each point. The real part $h$ defines a Riemannian metric $g$ on the underlying smooth manifold:
\[ g = \frac{1}{2} (h + \bar{h}). \]
The length, volume, and all the corresponding concepts in Riemannian geometry for a Hermitian manifold are with respect to this induced metric $g$. 

A differentiable map $f:N\to M$ between two complex manifolds is holomorphic if and only if $D f(i v)= i D f(v)$ for any $v\in TN$. A $C^1$ real submanifold $N$ of a complex manifold is called a \textit{complex submanifold} if it admits a complex structure such that the inclusion map is holomorphic, or, equivalently, $T_x N \subset T_x M$ is a complex subspace for any $x\in N$. 

In particular, the leaves of dynamically defined foliations are all complex manifolds. 
\begin{lemma}\label{leaf}
If the foliation $W^\ast$ ($\ast=s,u,c,cs,cu$) exists, then the leaves of $W^\ast$ are complex submanifolds of $M$.
\end{lemma}
\begin{proof}
    It follows from the definition that the leaves of $W^\ast$ are $C^1$ submanifolds of $M$. Since $f$ is holomorphic, for any $v\in E^\ast$, we have $\|D f(i v)\|=\|i D f(v)\|=\|D f(v)\|$. Consequently $i v \in E^\ast$, $E^\ast$ is a complex subspace and the leaves of $W^\ast$ are complex submanifolds.
\end{proof}
Let $M$ be a smooth manifold. An \textit{almost complex structure} $J$ on $M$ is a linear complex structure (that is, a linear map which squares to $-1$) on tangent space at each point of the manifold, which varies $C^\infty$-ly on the manifold. A manifold equipped with an almost complex structure is called an \textit{almost complex manifold}.

Every complex manifold induces a natural almost complex structure.
However, not every almost complex structure comes from a complex manifold. An almost complex structure $J$ that comes from the complex manifold is called \textit{integrable}. Such a $J$ is also called a \textit{complex structure}. The Newlander-Nirenberg theorem gives a criterion for whether an almost complex structure is integrable.

\begin{theorem}[Newlander-Nirenberg {\cite{nn57}}]
\label{nn}
An almost complex structure $J$ is integrable if and only if the Nijenhuis tensor
\[ N_J (X, Y) := [X, Y] + J ([J X, Y] + [X, J Y]) - [J X, J Y] \]
vanishes everywhere for any smooth vector fields $X$ and $Y$ on $M$.
\end{theorem}

\subsection{Topological facts about fibered partially hyperbolic system}\label{fphs}

A partially hyperbolic diffeomorphism $f$ is called \textit{dynamically coherent} if there exist a center-stable foliation $W^{cs}$ and a center-unstable foliation $W^{cu}$ tangent to the bundles $E^{cs}=E^c \oplus E^s$ and $E^{cu}=E^c \oplus E^u$, respectively.

\begin{proposition}[Theorem 1 in \cite{bb16}]
\label{coh}
Fibered partially hyperbolic diffeomorphisms are dynamically coherent.
\end{proposition}
If $f$ is dynamically coherent partially hyperbolic diffeomorphism, then each leaf of $W^{cs}$ is simultaneously subfoliated by $W^c$ and $W^s$. In particular, for any $x \in M$, $y \in W^s(x)$, there is a neighborhood $U(x) \subset W^c(x)$ of $x$ and a local homeomorphism $h^s_{xy} : U(x) \rightarrow W^c(y)$ such that $h^s_{xy}(z) \in W^s(z) \cap W^c_{\mathrm{loc}}(y)$. We call $h^s_{xy}$ a (local) stable holonomy map, and similarly, we can define the (local) unstable holonomy map $h^u$. Since the stable and unstable leaves are contractible, $h^{\ast}_{xy}$ ($\ast = s, u $) is well-defined and uniquely determined as a germ by $x$ and $y$.

We say that a dynamically coherent partially hyperbolic diffeomorphism $f$ admits \textit{global stable holonomy maps} if for any $x \in M$, $y \in W^s(x)$, there exists a globally defined homeomorphism $h^s_{xy} : W^c(x) \rightarrow W^c(y)$ such that $h^s_{xy}(z) = W^s(z) \cap W^c(y)$. Since global holonomy maps coincide locally with local holonomy maps, we use $h^s_{xy}$ to denote both local holonomy maps and global holonomy maps. Similarly, we can define global unstable holonomy maps $h^u$ and say that $f$ \textit{admits global $su$-holonomy maps} if $f$ admits both global stable and unstable holonomy maps. In particular, if $f$ admits global $su$-holonomy maps, then all leaves of $W^c$ are homeomorphic.

\begin{proposition}[Lemma 3.5 in \cite{avw22}]
  \label{holo}
  Fibered partially hyperbolic diffeomorphisms admit global $su$-holonomy maps.
\end{proposition}
\subsection{Bunching and holomophy of the stable and unstable holonomies}
Let $f$ be a partially hyperbolic diffeomorphism, we say that $f$ is \textit{center bunching} if there exists integer  $k \geq 1$ such that %for every $p \in M$,
\begin{eqnarray*}
&&\sup_p \|D f^k|_{E^s(x)}\|\cdot \|(D f^k|_{E^c(x)})^{-1}\|\cdot \|D_p f^k|_{E^c(x)}\|<1, \hbox{ and}\\
&&\sup_p \|(D f^k|_{E^u(x)})^{-1}\|\cdot \|D f^k|_{E^c(x)}\|\cdot \|(D f^k|_{E^c(x)})^{-1}\|<1.
\end{eqnarray*}
\begin{proposition}[\cite{psw97}]
    \label{1b}
    Let $f: M \to M$ be a dynamical coherent partially hyperbolic $C^2$ diffeomorphism. If $f$ is center bunching, then the local stable and unstable holonomy maps are uniformly $C^1$ when restricted to each center unstable and each center stable leaf, respectively.
\end{proposition}

Since $f$ is holomorphic, $Df|_{E^c}$ is conformal, and $f$ always satisfies the center bunching condition and the Proposition above applies. Moreover, we have the following stronger result : 

\begin{corollary}
    \label{sho}
    Let $f$ be a fibered holomorphic partially hyperbolic diffeomorphism. If $\dim_\mathbb C E^c=1$, then the local stable(unstable) holonomy maps are holomorphic when restricted to each center stable(unstable) leaf. In particular, $W^{s(u)}$ holomorphically subfoliates $W^{cs(cu)}$, all center leaves are holomorphic equivalent. 
\end{corollary}

\begin{proof}
  We prove the result for the stable holonomy, the proof for the
  unstable holonomy is similar by considering $f^{-1}$. By Proposition \ref{1b}, the local stable holonomy along center direction is uniformly $C^1$. For any $z \in W^c_{\mathrm{loc}}(x), w = h^s_{xy}(z)$, by differentiating the equation
  \[ h^s_{xy}(z) = f^{-n} \circ h^s_{f^n x f^n y} \circ f^n (z) \]
  we have
  \[ D h^s_{xy}(z) = (D f^n_w|_{E^c})^{-1} \circ D h^s_{f^n z f^n w} \circ
     D f^n_z|_{E^c}. \]
  Since $D f$ is conformal, and $D h^s_{f^n z f^n w}$ is close to $\mathrm{Id}$ as $n \to \infty$,  we see that $D h^s_{xy}(z)$ is conformal by letting $n \rightarrow \infty$. Therefore, $h^s_{xy}$ is holomorphic. By Lemma \ref{leaf} and Corollary \ref{hf}, $W^{s}$ holomorphically subfoliates $W^{cs}$. Recall that fibered partially hyperbolic systems have globally $su$-holonomy (see Proposition \ref{holo}); hence, all center leaves are holomorphic equivalent.
\end{proof}

\subsection{Disintegration of measures and Gibbs measures}

We first recall the disintegration of measures (see \cite{avw15,avw22} and \cite{bdv04} for more details). Let $M$ be a Polish space with a Borel probability measure $\mu$. Let $\mathcal{P}$ be a partition of $M$ into measurable subsets, and let $\hat{\mu}$ denote the induced measure on the $\sigma$-algebra generated by $\mathcal{P}$.%which can be regarded as a measure on $\mathcal{P}$. 

A \textit{disintegration} (or a \textit{system of conditional measures}) of $\mu$ with respect to the partition $\mathcal{P}$ is a family of probability measures $\{\mu_P\}_{P \in \mathcal{P}}$ on $M$ with the following properties:
\begin{itemize}
  \item $P$ is of full $\mu_P$-measure for $\mu$-almost every $P \in \mathcal{P}$;
  
  \item For any continuous function $\varphi: M \rightarrow \mathbb{R}$, the function $P \mapsto \int \varphi d \mu_P$ is measurable for the $\sigma$-algebra generated by $\mathcal{P}$, and
  \[ \int_M \varphi d \mu = \int_{\mathcal{P}} \left(\int_P \varphi d \mu_P\right) d \hat{\mu}(P) . \]
\end{itemize}

The disintegration may not exist for an arbitrary given $\mathcal{P}$, but it always exists if $\mathcal{P}$ is the union of local leaves of a foliation. In fact, Rokhlin~\cite{rok49} proved that there is a disintegration $\{\mu^{\mathcal{B}}_x : x \in \mathcal{B}\}$ of the restriction of $\mu$ to the foliation box into conditional measures along the local leaves. This disintegration is essentially unique in the sense that two such systems coincide in a set of full $\hat{\mu}$-measure. Moreover, if the leaves of the foliation are compact, then $\mathcal{P}$ can be chosen as the union of entire leaves:

\begin{proposition}[Proposition 3.7 in {\cite{avw22}}]
  Let $W$ be a foliation of $M$, and let $\mu$ be a Borel probability measure on $M$. If the leaf $W_x$ is compact for $\mu$-almost every $x\in M$, then there exists a system of conditional measures relative to $\{W_x: x\in M\}$ that is essentially unique. 
\end{proposition}
We then introduce some facts about Gibbs measures (see \cite{bdv04} for more details). Let $f \in \mathrm{Diff} (M)$ be a partially hyperbolic diffeomorphism. An
invariant probability measure $\mu$ is a \textit{Gibbs $u (s)$-state} if its conditional measures along the local unstable(stable) leaves are equivalent to the leafwise Lebesgue measures. 
\begin{proposition}[Theorem 11.8, Lemma 11.13 and Corollary 11.14 in \cite{bdv04}]
  \label{gbu}Let $f$ be a $C^2$ partially hyperbolic diffeomorphism, then 
  \begin{itemize}
      \item there exists at least one Gibbs $u$-state;
      \item  the ergodic components of a Gibbs $u$-state are also Gibbs $u$-states;
      \item the support of any Gibbs $u$-state consists of entire strong unstable leaves.
  \end{itemize} 
\end{proposition}
We also have the following definition of Gibbs $cu$-state:

\begin{definition}\label{gbcu}
    Let $f$ be a dynamically coherent partially hyperbolic diffeomorphism. An invariant probability measure $\mu$ is a \textit{Gibbs $cu$-state} if its conditional measures along local center-unstable leaves are equivalent to leafwise Lebesgue measures.
\end{definition}

Gibbs $cs$-state can be defined in a similar way. Generally speaking, a Gibbs $cu (cs)$-state may not exist without additional assumptions on $f$. We will construct Gibbs $cu (cs)$ respectively by the action of the heat kernel (see Proposition \ref{s5}).

\subsection{The regularity of foliations and maps}\label{journee}
We first recall the definition of foliations. Let $M$ be a manifold of dimension $n \geq 2$. A \textit{foliation} is a partition $F$ of the manifold $M$ into continuous submanifolds ($F$-leaves) of dimension $k$ for some $0 < k < n$. To be precise, for every $p \in M$, there exists a continuous homeomorphism
\[ \varphi : \mathbb D^k \times \mathbb D ^{n - k} \rightarrow M \quad (\mathbb D^m \text{ denotes the unit disk in } \mathbb{R}^m) \]
with $\varphi(0, 0) = p$, such that the image of every horizontal $\mathbb D^k \times \{ \eta \}$ is contained in some $F$-leaf.

A foliation is \textit{holomorphic} if there is a holomorphic foliation atlas. It follows from the holomorphic Frobenius' theorem (see Theorem 2.26 in \cite{voi02}) that a foliation is holomorphic if and only if its leaves is tangent to a holomorphic distribution. To get a criterion for checking whether a foliation with complex  leaves is holomorphic, we need to establish a holomorphic analogue of Journ\'{e}'s lemma. We first recall the classic $C^\infty$ Journ\'{e}'s lemma:

\begin{proposition}[\cite{jou88}]
  \label{rjou}
  Let $W_1$ and $W_2$ be transverse foliations of a manifold $M$ with uniformly $C^{\infty}$ leaves. Suppose $\psi:M \to \mathbb R^n$ is a continuous function such that restrictions of $\psi$ to the leaves of $W_j$ are uniformly $C^{\infty}$ for $j=1,2$ respectively, then $\psi$ is $C^{\infty}$.
\end{proposition}
We then state the holomorphic Journ\'{e}'s lemma. It's worth noting that we do not need the ``uniform" assumption in the holomorphic Journ\'{e}'s Lemma.
\begin{proposition}[Holomorphic Journ\'{e}'s Lemma]
  \label{cjou}
  Let $W_1$ and $W_2$ be transverse foliations of a complex manifold $M$ with complex leaves. Suppose $\psi:M \to \mathbb C^n$ is a complex-valued continuous function such that restrictions of $\psi$ to the leaves of $W_j$ are holomorphic for $j=1,2$ respectively, then $\psi$ is a holomorphic map.
\end{proposition}
\begin{proof}
$\psi$ is uniformly continuous on any given foliation chart. By Cauchy's integral formula, a uniformly $C^0$ holomorphic map is uniformly $C^\infty$ thus $\psi$ is uniformly $C^\infty$ along the leaves of $W_j$. It follows from Proposition \ref{rjou} that $\psi$ is $C^\infty$. By assumptions, $\psi$ is holomorphic along $W_j$, $D \psi (i v)= i D\psi (v)$ for any $v\in T W_j$ $(j=1,2)$. Since $W_1$ and $W_2$ are transverse and $\psi$ is $C^\infty$, we have $D \psi (i v)= i D\psi (v)$ for any $v\in T W_1+TW_2=T M$ and therefore $\psi$ is holomorphic.
\end{proof}
This has the following corollary:
\begin{corollary}
  \label{hf}For a complex manifold $M$, let $W$ and $F$ are two transverse foliations of $M$ such that both $W$ and $F$ have complex leaves. If the local holonomy maps along $W$ between any two $F$-leaves are holomorphic, then $W$ is a holomorphic foliation. 
\end{corollary}

\begin{proof}
The proof is a simple adaptation for the $C^\infty$ case (see for example \cite{psw97} and \cite{bx18}). Let $n= \dim_\mathbb C M$ and $k= \dim_\mathbb C W$. We fix a point $x \in M$ together with a neighborhood $V$ of $x$ and choose a holomorphic coordinate chart $g: V \to \mathbb C^k \times \mathbb C^{n-k}$ such that $g(V \cap W(x)) \subset \mathbb C^{k} \times \{0\}$ and $g(V \cap F(x)) \subset \{0\} \times \mathbb C^{n-k}$. We then define for $p=(y,z) \in g(V)$,
$$\Psi(p)=(y,g(W)(p) \cap g(F)(0))=(y, h_{p,0}(z))$$
where $g(W)$, $g(F)$ denote the images of our foliations under $g$ and $h_{p,0}(z)$ is the unique intersection point of $g(W)(p)$ with $g(F)(0)$ inside $g(V)$. This map straightens the $W$-foliation into a foliation of $\mathbb C^k \times \mathbb C^{n-k}$ by $k$-disks $\mathbb D^k \times \{z\}$. Since the leaves of $W$ are complex submanifolds, the map $\Psi$ is holomorphic when restricted to the leaves of $W$, and since the holonomy maps of the $W$ foliation between $F$-transversals are holomorphic the chart $\Psi$ is also holomorphic along the leaves of $F$. By Proposition \ref{cjou} this implies that $\Psi$ is holomorphic.
\end{proof}
\subsection{\red{Non-stationary linearization and the holonomy of $Df|_{E^u}$}}

We introduce the non-stationary linearization of $W^u$ in this subsection, all results apply to $W^s$ by considering $f^{-1}$. \red{The following proposition is essentially established in \cite{ghy95}:}

\begin{proposition}
  \label{affine}
  Let $f$ be a holomorphic partially hyperbolic diffeomorphism with $\dim_{\mathbb{C}} E^{u} = 1$. Then for each $x \in M$, there is a holomorphic diffeomorphism $\Phi_x : E^u_x \rightarrow W^{u}(x)$ satisfying
 \begin{enumerate}
    \item $\Phi_{fx} \circ Df_x = f \circ \Phi_x$;
    \item $\Phi_x(0) = x$  and $D_0\Phi_x$ is the identity map;
    \item the family of diffeomorphisms $\{ \Phi_x \}_{x \in M}$ varies continuously with $x$.
\end{enumerate}
The family of diffeomorphisms $\{ \Phi_x \}_{x \in M}$ satisfying the conditions above is unique.
\end{proposition}
The bundle $E^u$ is a Hölder continuous subbundle of $TM$ with some Hölder exponent $\beta > 0$. Therefore the restriction $Df|_{E^u}$ of the derivative of $f$ to the unstable bundle is a Hölder continuous linear cocycle over $f$ in the sense of Kalinin-Sadovskaya \cite{ks13}. For $x, y \in M$ two nearby points we let $I_{xy}: E^u_x \to E^u_y$ be a linear identification which is $\beta$-Hölder close to the identity. Since $f$ is holomorphic, $Df|_{E^u}$ is conformal (and is fiber bunched in the sense of \cite{ks13}). The following proposition thus applies to $Df|_{E^u}$.
\begin{proposition}[Proposition 4.2 in {\cite{ks13}}]
  \label{cocycle}For any $x \in M, y \in
  W^u(x)$, the limit\color{red}
  \[ H^u_{xy} \assign \lim_{n \rightarrow +\infty} Df^{n}_{f^{-n}y}|_{E^u}
     \circ I_{f^{-n} x f^{-n} y} \circ Df^{-n}_x|_{E^u} \]
\color{black}exists uniformly in $x$ and $y$ and defines a $\mathbb R$-linear map from $E^u_x$ to $E^u_y$ with the following properties:
  \begin{itemize}
    \item $H^u_{xx} = \mathrm{Id}$ and $H^u_{yz} \circ H^u_{xy} = H^u_{xz}$;
    \item $H^u_{xy} = Df^{n}_{f^{-n}y}|_{E^u} \circ H^u_{f^{-n} x f^{-n} y} \circ Df_x^{-n}$
    for any $n > 0$;
    \item $\| H^u_{xy} - I_{xy} \| \leqslant C d(x, y)^{\beta}$, where $\beta$ is the exponent of H\"{o}lder continuity for $E^u$.
  \end{itemize}
\end{proposition}
Similarly, if $y \in W^s_{\mathrm{loc}} (x)$, then the limit
\[  \lim_{n \rightarrow +\infty} (Df^{n}_y|_{E^u})^{-1}
     \circ I_{f^{n} x f^{n} y} \circ Df^{n}_x|_{E^u} \assign H^s_{x y} \]
exists and gives a linear map from $E^s_x$ to $E^s_y$ with analogous
properties. $H^u$ and $H^s$ are known as the \textit{unstable holonomy} and \textit{stable holonomy} for $D f|_{E^u}$ respectively. 
\begin{remark}
Since $D f|_{E^u}$ is $\mathbb C$-linear and $I_{x y}$ is close to a conformal linear map when $x$ and $y$ are close, $H^u$ and $H^s$ are actually $\mathbb C$-linear. 
\end{remark}
We end this section with a lemma we shall use later:
\begin{lemma}[Proposition 9 in {\cite{bx18}}]
  \label{l9}For each $x \in M$ and $y \in W^u(x)$, the map $\Phi^{-1}_y \circ \Phi_x :
  E^u_x \rightarrow E^u_y$ is an affine map with derivative $H^u_{xy}$.
\end{lemma}

\subsection{\red{Quasiconformality and center holonomy}\label{subsec qc}}
We first recall some classical results for quasiconformal maps on $\mathbb{C}$. Let $h \colon U \to V$ be a homeomorphism between two open subsets of $\mathbb{C}$. The linear dilatation of $h$ at a point $x \in U$ is defined as
\[
L_h(x) = \limsup_{r \to 0} \, \frac{\max_{\|y - x\| = r} \|h(y) - h(x)\|}{\min_{\|y - x\| = r} \|h(y) - h(x)\|}.
\]
For $K \geq 1$, the map $h$ is called \emph{$K$-quasiconformal} if $L_h(x) \leq K$ for all $x \in U$. When the constant $K$ is omitted, we simply refer to $h$ as a \emph{quasiconformal map}. The following lemma about quasiconformal map will be used throughout the text. 
\begin{lemma}[\cite{ahl06}]
  \label{qsl}
  Let $h:U \to V$ be a quasiconformal map between open domains of $\mathbb{C}$, then $h$ is absolutely continuous and differentiable Lebesgue almost everywhere. Moreover, if $\bar{\partial} \varphi = 0$ for Lebesgue-a.e. $z \in U$, then $\varphi$ is holomorphic.
\end{lemma}

Since $\dim_\mathbb C E^u= \dim_\mathbb C E^s=1$, $D f|_{E^u}$ and $D f|_{E^s}$ are conformal. In particular, $f$ is a uniformly quasiconformal partially hyperbolic diffeomorphism in the sense of \cite{bx18}, and we have: 

\begin{proposition}[Lemma 12 and Proposition 19 in {\cite{bx18}}]
  \label{cac}
  Let $f$ be a fibered holomorphic partially hyperbolic diffeomorphism of a compact complex 3-fold $M$, and let $\{\Phi_x\}$ be the unique family of holomorphic diffeomorphisms in Proposition \ref{affine}. Then for any $x \in M$, $y \in W^{cs}_{\mathrm{loc}}(x)$, 
  $$ \Phi_y^{-1} \circ h^{cs}_{\mathrm{loc}}\circ \Phi_x : \Phi_x^{-1} (W^u_{\mathrm{loc}}(x) )\to \Phi_y^{-1} (W^u_{\mathrm{loc}}(y) )$$
is quasiconformal. In particular, the center holonomy $h^c$ along $W^u$ is quasiconformal, absolutely continuous, and differentiable almost everywhere (with respect to the leafwise Lebesgue measure). 
\end{proposition}
For any $x \in M$, let $L_{\mathrm{loc}}^{cu}(x):=W^c(W^u_{\mathrm{loc}}(x))$ be the $c$-saturated set containing $W^u_{\mathrm{loc}}(x)$ and let $m^{cu}_x$ be the leafwise Lebesgue measure on $L_{\mathrm{loc}}^{cu}(x)$. Let $m^\ast$ be the leafwise Lebesgue measure along the $\ast$-leaves ($\ast= c, u$). As a consequence of the absolute continuity of the center holonomy, we have the following Fubini type theorem:
\begin{proposition}[Corollary 17 in \cite{bx18}]
  \label{fub}
 Let $f$ be a fibered holomorphic partially hyperbolic diffeomorphism of a compact complex 3-fold $M$, then
 \[ m_x^{c u} \asymp \int_{W^c (x)} m^u_y d m^c_x (y) \asymp \int_{W^u_{\mathrm{loc}} (x)} m^c_y d m^u_x (y). \]
where ``$\asymp$'' denoting that two measures are equivalent.
\end{proposition}
\section{A dichotomy: isometries or contractions}\label{sec dic}
\textbf{The following five sections are dedicated to the proof of Theorem \ref{t2}}. Throughout these five sections, $f$ is always a holomorphic fibered partially hyperbolic diffeomorphism on a compact complex 3-fold $M$, $h^c$ always denotes the center holonomy along the unstable direction, and $\mu$ is always an ergodic Gibbs $u$-state. Let $\hat{M} = M / W^c$ denote the leaf space of the center foliation. We will always use $\hat{\cdot}$ symbols for quotient objects. For example, $\hat{f}$ is the quotient map of $f$ on $\hat{M}$, $\hat{\mu}$ is the quotient measure on $\hat{M}$ of a measure $\mu$ on $M$, $W^c(x) = W^c(\hat{x})$ is the center leaf at $x$, etc.

A center-leaf-wise defined metric $g^c$ is called \textit{uniformly bounded} if $d_{g^c_{\hat x}}(x_1,x_2) \leq C d^c(x_1,x_2)$ is uniformly bounded for any $\hat x \in \hat M$ and $x_1,x_2 \in W^c(\hat x)$, where $d^c(\cdot, \cdot)$ is the distance induced by the restriction of a Riemannian metric on $M$ to the center leaf. Clearly, this definition is independent with the choice of the Riemannian metric. 

\begin{definition}
  If there is a fiberwise measurably defined Riemannian metric $g^c$ on $M$ such that
  \begin{itemize}
      \item $g^c$ is uniformly bounded;
      \item $f|_{W^c (\hat{x})}$ is an isometry for $\hat{\mu}$-almost every $\hat{x} \in \hat{M}$,
  \end{itemize}
  then $f$ is called a \textit{$\mu$-center isometry}.
\end{definition}
For any $x \in M$, we denote by $L_{\mathrm{loc}}^{cu}(x) \assign \{ y \in M | \hat{y} \in \hat W^u_{\mathrm{loc}}(x) \}$, the $W^c$-saturated set containing $W^u_{\mathrm{loc}}(x)$. 
\begin{definition}
    If there exists a subset $A \subset \mathrm{supp}\mu$ such that
  \begin{itemize}
    \item $ \liminf_{n \rightarrow + \infty} d^c (f^n x, f^n x') = 0 $ for any $x \in A$, $x' \in A \cap W^c (x)$;
    \item $A \cap L_{\mathrm{loc}}^{cu}(x)$ has full (leafwise Riemannian) volume in $L_{\mathrm{loc}}^{cu}(x)$ for $\mu$-a.e. $x$. 
    \end{itemize}
    then $f$ is called a \textit{$\mu$-center contraction}.
\end{definition}
The main result of this section is:

\begin{proposition}
  \label{dic}For any ergodic Gibbs $u$-state $\mu$, there is a dichotomy:
  \begin{itemize}
    \item either $f$ is a $\mu$-center isometry;
    \item or $f$ is a $\mu$-center contraction.
  \end{itemize}
\end{proposition}
Recall that all center leaves are holomorphic equivalent compact complex curves (see Proposition \ref{sho}). The ambient manifold $M$ is a continuous fiber bundle over $\hat M$ whose fibers are given by the center leaves. Denote \textit{$p_g$} as the genus of the center leaves; we first prove that

\begin{lemma}
  If $p_g > 0$, then $f$ is a $\mu$-center isometry.
\end{lemma}
\begin{proof}
  The lemma is a simple corollary of a classical result in complex analysis: If a compact Riemann surface has genus $p_g \geqslant 1$, then the holomorphic automorphism group preserves the metric of constant curvature. 
    To be precise, let $g_{\hat{x}}$ be the fiberwise defined metric with constant curvature and such that $\mathrm{vol}(W^c(\hat{x}))=1$ for each $\hat{x} \in \hat{M}$. $f|_{W^c (\hat{x})}$ is a holomorphic diffeomorphism and therefore is an isometry with respect to $g_{\hat{x}}$ and $g_{\hat{f} (\hat{x})}$, then $f$ is a $\mu$-center isometry by definition.
\end{proof}

The remaining case is that $p_g=0$, i.e., $M$ is a continuous $\mathbb{P}^1$ bundle over the leaf space $\hat{M}$. We fix a metric $g$ on $M$ and define:
\[ \hat{B} \assign \{ \hat{x} \in \hat{M} | \|Df^n|_{W^c(\hat x)}\|
   \text{ is uniformly bounded for } n > 0\}. \]
Roughly speaking, $\hat{B}$ consists of center leaves that exhibit isometric features. Note that since $M$ is compact, the definition of $\hat{B}$ does not rely on the choice of metric. Moreover, if we identify $f^n|_{W^c(\hat x)}$ with a matrix in the holomorphic automorphism group $\mathrm{Aut}(\mathbb P^1) \simeq \mathrm{PSL}(2,\mathbb{C})$, then $\|f^n|_{W^c(\hat x)}\|$ is uniformly bounded with respect to the matrix norm for any $\hat x \in \hat B$\footnote{Since $\hat M$ is compact, there exist finitely many local trivialization charts $(\hat U_i, \phi_i)$ of the fiber bundle $\pi :M \to \hat M$ covering $\hat M$. In particular, $\|Df^n|_{W^c(\hat x)}\|$ is uniformly bounded if and only if $\phi_{i_{\hat f^n(\hat x)}}^{-1}\circ f^n|_{W^c(\hat x)} \circ \phi_{i_{\hat x}} \in \mathrm{PSL}(2,\mathbb{C})$ is uniformly bounded, where $i_{\hat x}$ and $i_{\hat f^n(\hat x)}$ are subscripts such that $\hat x \in \hat U_{i_{\hat x}}$ and $ \hat f^n (\hat x) \in \hat U_{i_{\hat f^n(\hat x)}}$ respectively.}. We mention the following lemma from linear algebra which will be used later, whose proof can be find in Appendix \ref{sec liag}. 
\begin{lemma}\label{liag}
    Let $\{A_n\} \subset \mathrm{Aut}(\mathbb P^1) \simeq \mathrm{PSL}(2,\mathbb C)$ be a sequence of matrices such that $\|A_n\| \to \infty$ as $n$ tends to infinity. Then there is a sequence $\{n_k\} \subset \mathbb P^1$ and $b\in \mathbb P^1$ such that 
    $$\lim_{k \to \infty}d(A_{n_k} x_1,A_{n_k} x_2)=0$$
for any $x_1, x_2 \in \mathbb P^1 \backslash \{b\}$, where $d(\cdot, \cdot)$ is the distance function induced any Riemannian metric on $\mathbb P^1$. 
\end{lemma}
Since $\mu$ is ergodic, $\hat{B}$ is either a $\hat{\mu}$-null set or of full $\hat{\mu}$ measure. Proposition \ref{dic} follows immediately from the following two lemmas: 
\begin{lemma}
If $\hat \mu (\hat B)=1$, then $f$ is a $\mu$-center isometry.
\end{lemma}

\begin{proof}
The proof is inspired by the proof of Proposition 2.4 in \cite{ks10}. Note that $\pi : M \to \hat M$ is a continuous fiber bundle whose structure group is $\mathrm{PSL}(2,\mathbb C)$. We consider an associated fiber bundle $p:\mathcal{C} \to \hat M$ whose fibers are $\mathrm{PSL} (2, \mathbb{C}) / \mathrm{PSU}(2)$, sharing the same transition function with $\pi: M \to \hat M$, and the structure group $\mathrm{PSL}(2,\mathbb C)$ acts on $\mathcal{C}$ via
\begin{equation}\label{eqconf}
    X [C] = [ X^{\mathsf{H}} C X ] \quad \text{where } X,C \in \mathrm{PSL} (2, \mathbb{C}).
\end{equation}

It is known that $\mathrm{PSL} (2, \mathbb{C}) / \mathrm{PSU}(2)$ is a Riemannian symmetric space of non-positive curvature when equipped with a certain $\mathrm{PSL} (2, \mathbb{C})$-invariant metric $\rho$ (see \cite{hel78} for more details)\color{black}. The bundle map $f|_{W^c(\hat x)}: W^c(\hat x) \to W^c(\hat f \hat x)$ induces a pull-back action $(f|_{W^c(\hat{x})})^{\ast}:\mathcal C(\hat f \hat x) \to \mathcal C(\hat x)$ for any $\hat x \in \hat M$ via the action given by equation \eqref{eqconf}, which is an isometry.

Since a continuous fiber bundle is always measurably trivial via a bounded measurable identification, we can take $\tau_0$ to be a bounded measurable section of $\mathcal{C}$. For every $\hat{x} \in \hat{B}$, it follows from the definition of $\hat{B}$ that the set
\[ S(\hat{x}) = \{ (f^n|_{W^c(\hat{x})})^{\ast} (\tau_0(\hat{f}^n \hat{x})) : n > 0 \} \]
is bounded in $\mathcal{C} (\hat{x})$. Let $\omega(\hat{x})$ be the set of limit points of $S(\hat{x})$, which is non-empty, bounded and $f^\ast$-invariant. Since the space $\mathcal{C} (\hat{x})$ has non-positive curvature, for every $\hat{x} \in \hat{B}$ there exists a uniquely determined ball of the smallest radius containing $\omega(\hat{x})$. We denote its center by $\tau(\hat{x})$. It follows from the construction that the invariant section $\tau$ is a bounded measurable section that is $(f|_{W^c(\hat{x})})^*$-invariant.

Let $C: \hat{B} \to \mathrm{PSL}(2,\mathbb{C})$ be any bounded measurable map such that $C(\hat{x})$ represents $\tau(\hat{x})$ at every $\hat{x} \in \hat{B}$ (The non-uniqueness of $C$ does not affect our arguments). Let $g_{\hat{x}}$ be the fiberwise defined metric with constant curvature and unit volume. Then $C(\hat{f}(\hat{x}))^{-1} \circ f(\hat{x}) \circ C(\hat{x}) \in \mathrm{PSU}(2)$ is an isometry with respect to $g_{\hat{x}}$ and $g_{\hat{f}(\hat{x})}$ for every $\hat{x} \in \hat{B}$. Therefore, the pull-back metric $C^{\ast} g_{\hat{x}}$ is a fiberwise defined metric on $W^c(\hat{B})$ such that $f|_{W^c(\hat{x})}$ is an isometry for every $\hat{x} \in \hat{B}$. In particular, $f$ is a $\mu$-center isometry if $\hat \mu (\hat B)=1$.
\end{proof}

\begin{lemma}
  If $\hat{\mu} (\hat B) = 0$, then $f$ is a $\mu$-center contraction.
\end{lemma}

\begin{proof}By definition of $\hat B$, for any $\hat x \notin\hat{B}$, $\limsup_{n \rightarrow + \infty}\|Df^{n}|_{W^c(\hat x)}\| \to \infty$. Using local trivialization and applying Lemma \ref{liag}, we get for any $\hat x\notin \hat B$, there exists $b(\hat x) \in W^c(\hat x)$ such that for any $x, x' \in A_0 (\hat{x}) := W^c (\hat{x}) \backslash
   b(\hat x) $,  
  $$ \liminf_{n \rightarrow + \infty} d^c (f^{n} x, f^{n} x') = 0.$$ Let
  \[ A (\hat{x}) : = \bigcap_{n \in \mathbb{Z}} f^n A_0 (\hat{f}^{- n}
     \hat x) ,~~ A := \bigcup_{\hat{x} \in \hat{M} \backslash \hat{B}} A (\hat{x}).\]
Then $A$ is $f$-invariant, and $\liminf_{n \rightarrow + \infty} d^c \left( f^{n } {x, f^n}  x' \right) = 0 $ for any $x \in A, x' \in A \cap W^c (x)$.

We then prove that $A \cap L_{\mathrm{loc}}^{cu}(x)$ has full $m_x^{cu}$-measure for $\mu$-a.e. $x \in M$, where $m_x^{cu}$ denotes the leafwise Lebesgue measure. Let $W^c(A)$ be the $W^c$-saturated set containing $A$. It follows from the definition that $W^c(A) \setminus A$ intersects each center leaf in a Lebesgue null set. By Fubini's theorem (Proposition \ref{fub}), $A \cap L^{cu}_\mathrm{loc}(x)$ has full $m^{cu}_x$-measure if and only if $W^c(A) \cap L^{cu}_\mathrm{loc}(x)$ has full $m^{cu}_x$-measure.

Since $W^c(A) = M \setminus W^c(\hat{B})$ and $\hat{\mu}(\hat{B}) = 0$, the set $W^c(A)$ has full $\mu$-measure. Recall that $\mu$ is a Gibbs $u$-state, $W^c(A) \cap W^u_{\text{loc}}(x)$ has full $m^u_x$-measure for $\mu$-a.e. $x \in M$. Again by Fubini's theorem (Proposition \ref{fub}), $W^c(A) \cap L^{cu}_\mathrm{loc}(x)$ has full $m^{cu}_x$-measure for $\mu$-a.e. $x \in M$.\color{black}
\end{proof}

\section{The contracting case: the center holonomy is holomorphic}\label{sec contra}
In this section we prove the following proposition.
\begin{proposition}
  \label{d2}If $f$ is a $\mu$-center contraction, then for any $\hat x \in \mathrm{supp} \hat \mu$ and any $x, y \in W^c (\hat{x})$, the center holonomy between unstable leaves $h^{c}_{x y} : W^u_{\mathrm{loc}} (x) \rightarrow W^u_{\mathrm{loc}} (y)$ is holomorphic. 
\end{proposition}
The following lemma will be used in this section and later. For any $x \in M$ and $\delta>0$, denote $W^u_\delta(x)$ the unstable disk centered at $x$ with radius $\delta$. 
\begin{lemma}\label{l12 abst}
    Let $f$ be a holomorphic fibered partially hyperbolic diffeomorphism on a compact complex 3-fold. If $h^{c}_{xx'}:W^u_{\mathrm{loc}}(x) \to W^u_{\mathrm{loc}}(x')$ is differentiable at $x$ and there is a sequence of holomorphic map $\varphi_n:W^u_\delta(f^nx) \to M $ such that the $C^0$-distance
$$\liminf_{n \rightarrow +\infty} d_{C^0}(h^c_{f^nx f^nx'},\varphi_n) \to 0,$$
for some $\delta>0$, \color{black}then $\bar{\partial} (\Phi_{x'}^{-1}\circ h^c_{xx'}\circ \Phi_x)(0) = 0$. 
\end{lemma}

\begin{proof}
The strategy is to control the quasiconformal distortion on a small circle centered at $x$ inside $E^u_x$. Let $C^u_x(r_0)$ be the circle of radius $r_0$ in $E^u_x$ centered at $x$. For any $0<r_1 < r_2 \in \mathbb{R}$, let $R_x^u(r_1, r_2)$ be the closed annulus of radius $r_1, r_2$ centered at $x$.

By passing to a subsequence, we assume $\lim_{k \rightarrow +\infty} \red{d_{C^0}(h^c_{f^{n_k}x f^{n_k}x'},\varphi_{n_k})} \to 0.$ Let $\Phi_x : E^u_x \to W^u(x)$ be the holomorphic linearization given by Proposition \ref{affine}. Since $\varphi_{n_k}$ is holomorphic (conformal), for any $\epsilon>0$, there exists $r=r(\epsilon)$ and $K=K(\epsilon)$ such that for any $k\geq K(\epsilon)$, 
\begin{equation}\label{eqn: C0 close id}
(\Phi_{f^{n_k} x'}^{-1} \circ h^{c}_{f^{n_k} x f^{n_k} x'} \circ \Phi_{f^{n_k} x}) (C^u_{f^{n_k}(x)}(r))\subset R_{f^{n_k} x'}^u((1-\epsilon)r, (1+\epsilon)r).
\end{equation}
Using the fact $Df|_{E^u}$ is an expanding conformal linear map, we get there exists a subsequence  $r_k=r_k(x,\epsilon)\to 0$ such that 
$$Df^{n_k}_x|_{E^u}(C^u_x(r_k))=C^u_{f^{n_k}x}(r).$$
Therefore we have 
\begin{eqnarray*}
   && \Phi_{x'}^{-1} \circ h^{c}_{xx'} \circ \Phi_{x}(C^u_x(r_k)) \\
& = & (\Phi_{x'}^{-1} \circ f^{-n_k} \circ \Phi_{f^{n_k} x'}) \circ (\Phi_{f^{n_k} x'}^{-1} \circ h^{c}_{f^{n_k} x f^{n_k} x'} \circ \Phi_{f^{n_k} x}) \circ (\Phi_{f^{n_k} x}^{-1} \circ f^{n_k} \circ \Phi_{x})(C^u_x(r_k))\\
    & = & (Df_{x'}^{n_k}|_{E^u})^{-1} \circ (\Phi_{f^{n_k} x'}^{-1} \circ h^{c}_{f^n x f^{n_k} x'} \circ \Phi_{f^{n_k} x}) \circ Df_x^{n_k}|_{E^u}(C^u_x(r_k))\\
    &=& (Df_{x'}^{n_k}|_{E^u})^{-1} \circ (\Phi_{f^{n_k} x'}^{-1} \circ h^{c}_{f^{n_k} x f^{n_k} x'} \circ \Phi_{f^{n_k} x})(C^u_{f^{n_k}(x)}(r))\\
    &\subset & (Df_{x'}^{n_k}|_{E^u})^{-1}(R_{f^{n_k} x'}^u((1-\epsilon)r, (1+\epsilon)r)) \quad \text{by \eqref{eqn: C0 close id}}\\
    &\subset & R_{x'}^u((1-\epsilon)r'_k, (1+\epsilon)r'_k) \quad \text{for some $r'_k>0$, here we use conformality of $Df^{-1}|_{E^u}$.}
\end{eqnarray*}
In particular, for any $v_k, w_k \in C^u_x(r_k)$, we have$$\frac{\|\Phi_{x'}^{-1} \circ h^{c}_{xx'} \circ \Phi_{x}(v_k)\|}{\|\Phi_{x'}^{-1} \circ h^{c}_{xx'} \circ \Phi_{x}(w_k)\|} \le \frac{1+\varepsilon}{1-\varepsilon}.$$
By our assumption, $\Phi_{x'}^{-1} \circ h^{c}_{xx'} \circ \Phi_{x}$ is differentiable at $0 \in E^u$, therefore $D(\Phi_{x'}^{-1} \circ h^{c}_{xx'} \circ \Phi_{x})(0)$ is a $\frac{1+\varepsilon}{1-\varepsilon}$-quasiconformal linear map since $r_k\to 0$. Let $\epsilon\to 0$ we get $D(\Phi_{x'}^{-1} \circ h^{c}_{xx'} \circ \Phi_{x})(0)$ is 1-quasiconformal and thus $\mathbb{C}$-linear, i.e., $\bar \partial (\Phi_{x'}^{-1} \circ h^{c}_{xx'} \circ \Phi_{x})(0)=0$.
\end{proof}
As a corollary, we have
\begin{corollary}\label{l12}
    Let $f$ be a holomorphic fibered partially hyperbolic diffeomorphism on a compact complex 3-fold. If $h^{c}_{xx'}:W^u_{\mathrm{loc}}(x) \to W^u_{\mathrm{loc}}(x')$ is differentiable at $x$, and $\liminf_{n \rightarrow +\infty} d^c(f^n x, f^n x') = 0$ for some $x \in M$ and $x' \in W^c(x)$, then $\bar{\partial} (\Phi_{x'}^{-1}\circ h^c_{xx'}\circ \Phi_x)(0) = 0$. 
\end{corollary}
\begin{proof}
    Take $\varphi_n$ to be the identity map on $W^u_\mathrm{loc}(f^nx)$ and apply Lemma \ref{l12 abst}.
\end{proof}
\color{black}
\begin{lemma}
  \label{4l2} If $f$ is a $\mu$-center contraction, then for $\mu$-a.e. $x\in M$, $h^{c}_{x x'}:W^u_{\mathrm{loc}} (x) \to W^u_{\mathrm{loc}} (x')$ is holomorphic for $\mathrm{vol}_{W^c}$-a.e. $x' \in W^c (x)$.
\end{lemma}

\begin{proof}
Recall that if $f$ is a $\mu$-center contraction, then there exists a subset $A \subset \mathrm{supp}\mu$ such that
  \begin{itemize}
    \item $ \liminf_{n \rightarrow + \infty} d^c (f^n x, f^n x') = 0 $ for any $x \in A$, $x' \in A \cap W^c (x)$;
    \item $A \cap L^{cu}_\mathrm{loc}(x)$ has full (leafwise Riemannian) volume in $L_{\mathrm{loc}}^{cu}(x)$ for $\mu$-a.e. $x$. 
    \end{itemize}
By Fubini's theorem (Proposition \ref{fub}), $A \cap W^u_{\mathrm{loc}} (x')$ is also of full Lebesgue measure for $\mu$-a.e. $x\in M$ and $\mathrm{vol}_{W^c}$-a.e. $x' \in  W^c(x)$. Since $h^{c}_{x x'}$ is absolutely continuous (see Proposition \ref{cac}), if $A \cap W^u_{\mathrm{loc}} (x')$ is of full-Lebesgue measure, then $h^{c}_{x'x}(A \cap W^u_{\mathrm{loc}} (x'))\subset W^u_{\mathrm{loc}}(x)$ is also of full-Lebesgue measure. Moreover, by Proposition \ref{cac}, $h^{c}_{x x'}$ is differentiable for Lebesgue-a.e. $y \in W^u_{\mathrm{loc}}(x)$. Therefore,
\[ A_{xx'}^u \assign \{ y \in W^u_{\mathrm{loc}} (x)|y \in A, y'=h^{c}_{x x'}(y) \in A, h^{c}_{x x'}
     \text{ is differentiable at } y \text\} \]
is an intersection of three subsets of full Lebesgue measure and is of full Lebesgue measure for $\mu$-a.e. $x\in M$ and $\mathrm{vol}_{W^c}$-a.e. $x'\in W^c(x)$.

Now for any $y \in A_{xx'}^u$, $y' = h^{c}_{xx'}(y)$, $h^c_{xx'}$ is differentiable at $y$. Since $h^c_{xx'}$ coincides with $h^c_{yy'}$ near $y$, we have
\begin{eqnarray*}
\bar{\partial}(\Phi_{x'}^{-1} \circ h^{c}_{xx'} \circ \Phi_{x})(\Phi_x^{-1}(y))& = & \bar{\partial} (\Phi_{x'}^{-1} \circ \Phi_{y'}\circ \Phi_{y'}^{-1} \circ h^c_{yy'}\circ \Phi_y)(0)\\
(\text{by Lemma }\ref{l9})& = & \bar{\partial} (H^u_{y' x'} \circ   (\Phi_{y'}^{-1} \circ h^c_{yy'}\circ \Phi_y))(0)\\
(\text{by Corollary } \ref{l12})& = & 0.
\end{eqnarray*}
By Lemma \ref{qsl} and quasiconformality of center holonomy (Proposition \ref{cac}), this implies $\Phi_{x'}^{-1}\circ h^{c}_{xx'}\circ \Phi_x$ is holomorphic, then $h^c_{xx'}$ is holomorphic since $\Phi_x$ and $\Phi_{x'}$ are holomorphic.
\end{proof}
\begin{proof*}{Proof of Proposition \ref{d2}}
By Lemma \ref{4l2}, for $\mu$-a.e.\ $x\in M$, $h^{c}_{x x'}$ is holomorphic for $\mathrm{vol}_{W^c}$-a.e.\ $x' \in W^c (x)$. Since the $C^0$-limit of holomorphic functions is also holomorphic, this implies that $h^{c}_{x x'}$ is holomorphic for all $x' \in W^c (x)$. For any $\hat x \in \operatorname{supp} \hat \mu$ and any $x, y \in W^c (\hat{x})$, we can find sequences $x_n \to x$, $y_n \to y$ with $y_n \in W^c (x_n)$ such that $h^{c}_{x_n y_n}$ is holomorphic. Then $h^{c}_{xy}$ is holomorphic by taking the $C^0$-limit.
\end{proof*}

\section{The isometric case: construction of Gibbs $cu$-state}\label{sec gib}
In this section, we construct the Gibbs $cu$-state. The construction does not essentially rely on the holomorphic or quasiconformal assumption and applies to a large class of fibered partially hyperbolic systems.
\begin{proposition}
  \label{s5}
Let $f$ be a $C^\infty$ fibered partially hyperbolic diffeomorphism of a compact Riemannian manifold $M$, and let $\mu$ be an ergodic Gibbs $u$-state. If $f$ is a $\mu$-center isometry and $W^c$ is absolutely continuous within $W^{cu}$, then there exists a Gibbs $cu$-state $\nu$ (see Definition \ref{gbcu}) whose support is the $c$-saturated subset of $M$ projecting to $\mathrm{supp} \hat{\mu}$ in $\hat{M}$.
\end{proposition}
We begin with a Gibbs $u$-state $\mu_0:=\mu$, then the dual action of the heat semigroup on the measure space gives a family of measures $\mu_t, t \geq 0$. Since $f$ is $\mu$-center isometric, the center Lyapunov exponent vanishes, and we can apply the Avila-Viana invariant principle to deduce that $\mu_t$ has product structures in the $cu$-plaques. Finally, we take a weak$^\ast$- limit point $\nu$ of $\mu_t$ to get a Gibbs $cu$-state. Roughly speaking, $\nu$ is absolutely continuous along $W^c$ by properties of the heat semigroup action, and $\nu$ is absolutely continuous along $W^u$ since it is constructed from a Gibbs $u$-state and center holonomy between $W^u$ leaves are absolutely continuous (which is the key point of the proof).

\subsection{Action of the heat semigroup}\label{subsec heat}
\red{Before going into the proof, we briefly recall some facts about the heat semigroup on manifolds (see {\cite{gri09}} for more details).} Let $(N, g)$ be a closed \red{(compact and without boundary)} Riemannian manifold, $\mathrm{vol}$ be the normalized volume induced by $g$, and $\Delta$ be the Laplace-Beltrami operator associated with $g$. Then there is a semigroup action $P_t : L^2 (N, \mathrm{vol}) \rightarrow L^2 (N, \mathrm{vol})$
for $t > 0$ such that for any $\varphi\in L^2 (N, \mathrm{vol})$,  $u(t, x):=P_t(\varphi)$ is the solution to the Cauchy problem:
\[ \{\begin{array}{l}
     \frac{\partial u}{\partial t} = \Delta_x u, \quad t > 0\\
     u|_{t = 0} = \varphi.
   \end{array} \]

We mention some properties of $P_t$ which will be used later:
\begin{itemize}
  \item $P_t(1) = 1$ for any $t >0$. For any non-negative function $\varphi \in L^2(N)$, $P_t(\varphi) (x) \ge 0$ for any $t > 0, x \in N$. 
  \item If $\gamma : N \rightarrow N'$ is an isometry, then for any $\varphi \in L^2(N')$ 
\begin{equation}\label{he1}
    (P_{N, t}) (\varphi \circ \gamma) = P_{N', t} (\varphi) \circ \gamma
  \end{equation}
  \item For any $\varphi \in L^2(N), x \in N$, 
\begin{equation}\label{he2}
    \lim_{t \rightarrow + \infty} P_t(\varphi) (x) = \int_N \varphi (x) d \mathrm{vol}
  \end{equation}
converges to \red{the average of $\varphi$}. 
\end{itemize}

We turn to the proof now. Recall that $f$ being a $\mu$-center isometry implies that $f|_{W^c (\hat{x})}$ is an isometry with respect to a fiberwise defined Riemannian metric $g_{\hat x}$ for $\hat{\mu}$-a.e. $\hat{x} \in {\hat{M}}$. Let $P_{\hat{x}, t}$ be the heat semigroup on $W^c (\hat{x})$ with respect to $g_{\hat x}$. For any $t > 0$ and any Borel probability measure $\mu$, let $\mu_t$ be the linear functional acting on the space of bounded measurable functions such that for any bounded measurable function $\varphi$:
\[ \int_M \varphi d \mu_t = \int_{\hat{M}} \left( \int_{W^c (\hat{x})} P_{\hat{x}, t} (\varphi |_{W^c (\hat{x})}) d \mu^c_{\hat{x}} \right) d \hat{\mu} (\hat{x}), \]
where $\mu^c_{\hat{x}}$ is the conditional measure of $\mu$ along $W^c (\hat{x})$, and $\hat \mu$ is the projection of $\mu$ to $\hat M$. In particular, we have $\mu_0 = \mu$, and projections of  $\mu_t$ to $\hat M$ are always $\hat \mu$.

\begin{lemma}
 For any $t \geq 0$, $\mu_t$ is an $f$-invariant probability measure.
\end{lemma}
\begin{proof}
Since $P_{\hat x,t}(1) = 1$ and $P_{\hat x,t}(\varphi) (x) \ge 0$ for any $t >0$, any non-negative function $\varphi \in L^2(W^c(\hat x))$ and any $x \in W^c(\hat x)$, $\mu_t$ is a probability measure. Note that $f|_{W^c (\hat{x})}$ is an isometry for $\hat{\mu}$-a.e. $\hat{x} \in {\hat{M}}$, by equation $\eqref{he1}$ above, 
$$  P_{\hat{x}, t} (\varphi \circ f|_{W^c (\hat{x})})  = P_{\hat f\hat{x}, t} (\varphi |_{W^c (\hat f\hat{x})}) \circ f|_{W^c (\hat{x})} ~\text{ for } \hat{\mu}\text{-a.e. } \hat{x} \in {\hat{M}}.$$
Since $\mu$ is $f$-invariant, we have $(f|_{W^c(\hat x)})_\ast \mu^c_{\hat x}= \mu^c_{\hat f \hat x}$ for $\hat \mu$-a.e. $\hat x \in \hat M$ and 
\begin{eqnarray*}
    \int_{W^c (\hat{x})}P_{\hat{x}, t} (\varphi \circ f|_{W^c (\hat{x})}) d \mu^c_{\hat{x}}& = &  \int_{W^c (\hat{x})} P_{\hat f\hat{x}, t} (\varphi |_{W^c (\hat f\hat{x})}) \circ f|_{W^c (\hat{x})} d \mu^c_{\hat{x}} \\
 (\text{definition of push-forward measure)}  & = & \int_{W^c (\hat f\hat{x})} P_{\hat f\hat{x}, t} (\varphi |_{W^c (\hat f\hat{x})}) d (f_\ast\mu^c_{\hat{x}})\\
  (\text{invariance of } \mu^c_{\hat{x}}) & = &  \int_{W^c (\hat f\hat{x})} P_{\hat f\hat{x}, t} (\varphi |_{W^c (\hat f\hat{x})}) d \mu^c_{\hat f\hat{x}}.
  \end{eqnarray*}
Denote $\psi_{t}(\hat x) = \int_{W^c(\hat{x})} P_{\hat{x}, t} (\varphi|_{W^c(\hat{x})}) \, d\mu^c_{\hat{x}}$, then $\int_M \varphi \, d\mu_t = \int_{\hat{M}} \psi_t \, d\hat{\mu}$. Taking into account that $\hat{\mu}$ is $\hat{f}$-invariant, we obtain:
\begin{eqnarray*}
    \int_M \varphi \circ f ~d \mu_t & = & \int_{\hat{M}} \left( \int_{W^c (\hat{x})} P_{\hat{x}, t} (\varphi \circ f|_{W^c (\hat{x})}) ~d \mu^c_{ \hat{x}} \right) d \hat{\mu} (\hat{x})\\
(\text{by equation above)}   & = & \int_{\hat{M}} \left(  \int_{W^c (\hat f\hat{x})} P_{\hat f\hat{x}, t} (\varphi |_{W^c (\hat f\hat{x})}) d \mu^c_{\hat f\hat{x}}\right) d \hat{\mu} (\hat{x})\\
   (\hat{\mu} \text{ is } \hat{f}\text{-invariant}) & = & \int_{\hat{M}} \psi_t(\hat f\hat x)d \hat{\mu} (\hat{x})=\int_{\hat{M}} \psi_t(\hat x)d \hat{\mu} (\hat{x})=\int_M \varphi d \mu_t.
  \end{eqnarray*}
\end{proof}
\subsection{Invariance principle and construction of Gibbs $cu$-state}
Before presenting the proof, we state a version of the invariance principle by Avila-Viana \cite{av10} concerning extremal center Lyapunov exponents and invariant measures. While more general formulations exist, we state here a simplified version sufficient for our purposes.

\begin{proposition}[Corollary 4.3 in \cite{av10}]
    \label{ip}
Let $f: M \to M$ be a fibered partially hyperbolic diffeomorphism on a compact complex 3-fold and $\mu$ an $f$-invariant probability measure. If $f$ is a $\mu$-center isometry, then the conditional measures of $\mu$ along the center foliation are essentially $u$-invariant, i.e., $\mu$ admits a disintegration such that 
    $$(h^u_{xy})_*\mu_x^c = \mu_y^c \quad \text{for all} \quad y \in W^u(x)$$ 
for $\mu$-a.e. $x\in M$.
\end{proposition}

We turn to the proof now. For any $x\in M$, let $L_{\mathrm{loc}}^{cu}(x):=W^c(W^u_{\mathrm{loc}}(x))$ be the $W^c$-saturated set containing $W^u_{\mathrm{loc}}(x)$, $L_{\mathrm{loc}}^{cu}(x)$ is homeomorphic to $W^c(x) \times W^u_{\mathrm{loc}}(x)$.  Denote by $\Pi^{c}$ and $\Pi^u$ the  projections from the \textit{rectangle} $L_{\mathrm{loc}}^{cu}(x)$ to $W^c(x)$ along local unstable leaves and to $W^u_{\mathrm{loc}}(x)$ along local center leaves respectively. From now on we identify $L_{\mathrm{loc}}^{cu}(x)$ with $W^c(x) \times W^u_{\mathrm{loc}}(x)$ in the canonical way (through center and unstable holonomies). 

\begin{lemma}
  \label{product}
  For any $t \geq 0$, the local disintegration
  $\mu_{t,x}^{c u}$ of $\mu_t$ on $L_{\mathrm{loc}}^{cu}(x)$ has a product structure with respect to the topological
  product structure $L_{\mathrm{loc}}^{cu}(x)=W^c (x) \times W^u_{\mathrm{loc}} (x)$, i.e.
  \[ \mu_{t, x}^{c u} = \mu_{t, x}^c \times \Pi^u_{\ast} \mu_{ x}^{c u} \]
  holds for $\mu_t$-a.e. $x \in M$ (note that the second factor on the right hand side does not rely on $t$).
\end{lemma}

\begin{proof}
  Since $f$ is an $\mu$-isometry and $\hat{\mu}_t=\hat \mu$,  $f$ preserves a measurably fiberwise defined bounded metric for $\mu_t$ almost every center fiber. Therefore the center Lyapunov exponent of $\mu_t$ vanishes. 
  By Proposition \ref{ip}, $\mu_t$ admits a disintegration which is
  $u$-invariant on a full $\mu_t$-measure set $E_t$,  i.e., for any
  $ x \in E_t, y \in  W^u (x)$,
  \[ (h^u_{x y})_{\ast} \mu_{t, {x}}^c = \mu_{t, {y}}^c . \]
For any $x \in E_t$ and measurable sets $I_c \subset W^c(x)$ and $I_s \subset W^u_{\mathrm{loc}}(x)$, we have
 \begin{eqnarray*}
    \mu_{t, x}^{c u} (I_c \times I_s) & = & \int_{W^u_{\mathrm{loc}}(x)}  \mu_{t,y}^c (h^u_{xy}(I_c)) ~  d
     (\Pi^u_{\ast} \mu_{t, x}^{c u})(y)\\
    & = & \int_{W^u_{\mathrm{loc}}(x)}  ((h^u_{x y})_{\ast} \mu_{t, {x}}^c)  (h^u_{xy}(I_c)) ~  d
     (\Pi^u_{\ast} \mu_{t, x}^{c u})(y)\\
     & = & \int_{W^u_{\mathrm{loc}}(x)}  \mu_{t, x}^c (I_c) ~  d
     (\Pi^u_{\ast} \mu_{t, x}^{c u})(y)\\
    & = & \mu_{t, x}^c (I_c) \times
     (\Pi^u_{\ast} \mu_{t, x}^{c u}) (I_s).
  \end{eqnarray*}
Therefore $\mu_{t, x}^{c u} = \mu_{t, x}^c \times
     \Pi^u_{\ast} \mu_{t, x}^{c u}$ holds for $\mu_t$-a.e. $x \in M$. 

Finally, by the definition of  $\mu_t$, $\mu_t$ takes the same value on any $W^c$-saturated set for any $t \geq 0$, therefore
$$(\Pi^u)_{\ast} \mu_{t, x}^{c u}=(\Pi^u)_{\ast} \mu_{0, x}^{c u}=(\Pi^u)_{\ast} \mu_{ x}^{c u}$$
is independent of $t$. 
\end{proof}

Since $\mu_0= \mu$ is a Gibbs $u$-state, we have the following corollary:

\begin{corollary}\label{coro: pi u equiv}
 $\Pi^u_{\ast} \mu_{ x}^{c u} $ is equivalent to the leafwise Lebesgue measure along $W^u_{\mathrm{loc}}(x)$ for $\mu$-a.e. $x\in M$.
\end{corollary}

\begin{proof}
Let $m_x^u$ be the Lebesgue measure on $W^u_{\mathrm{loc}}(x)$. Since $\mu$ is a Gibbs $u$-state, $\mu^u_x \asymp m^u_x$ for $\mu$-a.e. $x \in M$. We first prove $\Pi^u_{\ast} \mu_{x}^{cu} \ll m_x^u$ for $\mu$-a.e. $x \in M$.

Let $S_x \subset W^u_{\mathrm{loc}}(x)$ be a Lebesgue null set, and $S=W^c(S_x)$ be the $c$-saturated set containing $S_x$. Since $h^{c}$ is absolutely continuous along $W^u$, $S_y:= S \cap W^u_{\mathrm{loc}}(y)=h^{c}_{xy}(S_x)$ is also a Lebesgue null set for any $y\in W^c(x)$. Therefore, 
\[ \mu^{cu}_{x}(S)= \int_{W^c(x)} \mu^u_y(S_y) ~d (\Pi^c_{\ast} \mu_{ x}^{c u})(y) =0. \]
By Lemma \ref{product}, we also have
\[ \mu_{ x}^{c u} (S)= \mu_{ x}^c(W^c(x)) \times (\Pi^u_{\ast} \mu_{ x}^{c u})(S_x)=\Pi^u_{\ast} \mu_{ x}^{c u}(S_x), \]
Therefore, $\Pi^u_{\ast} \mu_{ x}^{c u}(S_x)=0$, i.e., $\Pi^u_{\ast} \mu_{x}^{cu} \ll m_x^u$ for $\mu$-a.e. $x\in M$. 

Similarly, let $S_x \subset W^u_{\mathrm{loc}}(x)$ be a $\Pi^u_{\ast} \mu_{x}^{cu}$ null set, and let $S=W^c(S_x)$. Since $h^c$ is absolutely continuous along $W^u$, we have
$$\int_{W^c (x)} m^u_y (S \cap W^u_{\mathrm{loc}} (y)) d m^c_x (y)=0.$$
Recall that Proposition \ref{fub} says that: 
\[ \int_{W^c (x)} m^u_y d m^c_x (y)  \asymp \int_{W^u_{\mathrm{loc}} (x)} m^c_y d m^u_x (y), \]
thus $0= m^u_x(S_x)$ holds for $\mu$-a.e. $x\in M$. Therefore, $\Pi^u_{\ast} \mu_{x}^{cu} \asymp m_x^u$ for $\mu$-a.e. $x \in M$.
\end{proof}

Let $\nu$ be a weak$^\ast$ limit point of $\mu_t,t\in \mathbb R^+$, then we have:

\begin{proposition}
  \label{nac} $\nu$ is a Gibbs $c u$-state with $\mathrm{supp}
  \nu = \{ x \in M| \hat{x} \in \mathrm{supp} \hat{\mu} \}$.
\end{proposition}

\begin{proof}
  Let $\mu_{t_k} $ be a subsequence weak$^\ast$ converging to $\nu$. By Lemma \ref{product}, for $\nu$-a.e. $x$, the local disintegration of $\nu$ on $L_{\mathrm{loc}}^{cu}(x)$ satisfies
  $$\nu_{x}^{cu} = \nu_{x}^c \times \Pi^u_{\ast} \mu_{x}^{cu}.$$
Notice that here the second factor only depends on $\mu$, and is equivalent to the leafwise Riemannian volume on $W^u_{\mathrm{loc}}(x)$  (by Corollary \ref{coro: pi u equiv}). Then by Fubini's theorem (Proposition \ref{fub}),  to prove Proposition \ref{nac}, it suffices to show $\nu^c_x= \mathrm{vol}_{W^c (x)}$. For any $x \in M$ and any continuous function $\varphi_{x} : W^c (x) \rightarrow \mathbb{R}$, 
  \begin{eqnarray*}
    \int_{W^c (x)} \varphi_{x} d \nu^c_{x} & = & \lim_{k
    \rightarrow \infty} \int_{W^c (x)} \varphi_{x} d \mu^c_{t_k,
    x}=\lim_{k \rightarrow \infty} \int_{W^c (x)} (P_{x, t_k}
    \varphi_{x} )d \mu^c_{x}\\
   \text{by } \eqref{he2}  & = & \int_{W^c (x)} \left(
    \int_{W^c (x)} \varphi_{x} d \mathrm{vol}_{W^c (x)}
    \right) d \mu^c_{x}\\
    & = & \int_{W^c (x)} \varphi_x d \mathrm{vol}_{W^c (x)}.
  \end{eqnarray*}
Therefore $\nu^c_{x} = \mathrm{vol}_{W^c (x)}$ and $\nu$ is a Gibbs $cu$-state. The claim for the support of $\nu$ is a direct consequence of the definition of $cu$-state. 
\end{proof}

\section{The isometric case: the center holonomy is $\mathbb{R}$-linear}\label{sec con}

Recall that by Proposition \ref{cac}, $W^c$ is always absolutely continuous within $W^{cu}$ for a holomorphic fibered partially hyperbolic system $f$ on a compact complex 3-fold $M$. Therefore, Proposition \ref{s5} applies, and we have a Gibbs $cu$-state $\nu$ when $f$ is $\mu$-center isometric for some ergodic Gibbs $u$-state $\mu$. Recall that the normal form on $W^u$ is a family of holomorphic diffeomorphisms $\Phi_x: E^u_x \to W^u(x), x\in M$ (see Proposition \ref{affine}). The center holonomy is then linear on $W^c(\mathrm{supp}\mu)$ under the coordinate given by $\Phi_x$:

\begin{proposition}
  \label{reli}
  Let $\nu$ be a Gibbs $c u$-state, then for any $x, y \in
  \mathrm{supp} \nu, y \in W^c (x)$, $H^c_{xy}:=\Phi_y^{-1} \circ h^{c}_{x y} \circ \Phi_x$ is an $\mathbb R$-linear map. In particular, $h^{c}_{x y}$ is defined on the whole $W^u(x)$. 
\end{proposition}

The strategy for proving the linearity of center holonomy follows from \cite{bx18}. Since we do not assume $f$ is volume-preserving, we use the Gibbs $cu$-state $\nu$ we constructed to replace the invariant volume used in \cite{bx18}. This brings no essential difficulty since the disintegration of $\nu$ along $W^c$, $W^u_{\mathrm{loc}}$, and $W^{cu}_{\mathrm{loc}}$ is equivalent to the leafwise Lebesgue measures. Since the proof is very similar to that in \cite{bx18}, we put the proof in Appendix \ref{proof reli}. 

\begin{remark}\label{rs5m}
The proof of  Proposition \ref{reli} does not essentially rely on the holomorphic assumption of $f$. To prove the $\mathbb R$-linearity of the center holonomy maps, it suffices to assume $f$ is a $C^\infty$ uniformly quasiconformal fibered partially hyperbolic diffeomorphism and is a $\mu$-isometry. We leave the details to the interested reader. 
\end{remark}
\section{The isometric case: the center holonomy is holomorphic}\label{sec d1}
In this section, we prove the following key proposition: 
\begin{proposition}\label{d1}
    If $f$ is a $\mu$-center isometry, then for any $x \in M$, $\hat{x} \in \mathrm{supp} \hat{\mu}$, $y \in W^c (x)$, then the center holonomy along the unstable direction
    $h^{c}_{x y} : W^u_{\mathrm{loc}} (x) \rightarrow W^u_{\mathrm{loc}} (y)$
    is holomorphic. 
\end{proposition}
The proof relies on a well-behaved fiberwise translation map $T: W^{cu}(x) \to W^{cu}(x)$ along $W^u$. Although we do not know if $h^c$ is holomorphic a priori, we can still show that the map $T: W^{cu}(x) \to W^{cu}(x)$ is actually a holomorphic diffeomorphism. Consequently, the corresponding quotient manifold of the center-unstable leaf is a holomorphic elliptic fibration. Note that any holomorphic map from $W^c(x)$ to $\mathrm{Mod}(\mathbb{T}^2) = \mathbb{H} / \mathrm{SL}(2, \mathbb{Z})$ is trivial, which forces the elliptic fibration to be a fiber bundle, i.e., all fibers are holomorphically equivalent. In particular, this implies that the center foliation is holomorphic.

Take an arbitrary $x \in M$ with $\hat{x} \in \mathrm{supp} \mu$. Recall that $\Phi_x : E^u_x \to W^u(x)$ is a holomorphic diffeomorphism, it gives a translation structure on $W^u(x)$, i.e., for any $v\in E^u_x$, there is a translation map on $W^u(x)$ given by 
$$p \mapsto \Phi_x (\Phi_x^{-1} (p)
+ v) , p \in W^u(x),$$
which is a holomorphic automorphism of $W^u(x)$. By Proposition \ref{reli}, the center holonomy along the unstable direction is $\mathbb R$-linear under the chart $\Phi_x$, i.e. $\Phi_y^{-1} \circ h^c_{x y}\circ \Phi_x= H^c_{xy}$ is $\mathbb R$-linear for any $y\in W^c(x)$. Therefore we can extend translation map along $W^u(x)$ to a fiberwise translation map on $W^{cu}(x)$ via center holonomy maps. To be precise, for any $v\in E^u_x$, there is a homeomorphism 
\[ T_v : W^{c u} (x) \rightarrow W^{c u} (x), \quad p \mapsto \Phi_y (\Phi_y^{-1} (p)
   + H^{c}_{x y} (v)) \]
where $y = W^u (p) \cap W^c (x)$. It follows the definition that $T_v$ is a translation when restricted to the unstable leaves and is holomorphic along $W^u$. Moreover, we prove that: 

\begin{lemma}
For any $v \in E^u_x$, $T_v : W^{c u} (x) \rightarrow W^{c u} (x)$ is a holomorphic diffeomorphism. 
\end{lemma}

\begin{proof}
 Since $T_v$ is holomorphic along $W^u$, by the holomorphic Journ\'{e}'s lemma (Proposition \ref{cjou}), it suffices to prove that $T_v$ is holomorphic along $W^c$. For any $p\in W^{cu}(x)$, we will show that $T_v|_{W^c (p)} : W^c
  (p) \rightarrow W^{c u} (x)$ coincides with the unstable holonomy $h^u_{p T_v(p)} : W^c (p) \rightarrow W^c (T_v(p))$, which is holomorphic (see Proposition \ref{sho}). 
 
\red{By Proposition \ref{holo}, the unstable holonomy along center fibers is globally defined.} For any $p' \in W^c (p)$, let $y =
  W^u (p) \cap W^c (x)$ and $y' = W^u (p') \cap W^c (x)$. By Proposition \ref{reli}, we have
\[ \Phi_{y'} \circ H^{c}_{y  y'} = h^{c}_{y y'} \circ \Phi_y . \]
 A direct calculation gives:
  \begin{eqnarray*}
    h^u_{p T_v(p)} (p') & = & W^u (p') \cap W^c (T_v  (p)) = h^{c}_{p p'} (T_v(p))=h^{c}_{y y'} (T_v(p))\\
    & = & h^{c}_{y y'} \circ \Phi_y (\Phi_y^{-1}(p) + H^{c}_{x y} (v))\\
    & = & \Phi_{y'} \circ H^{c}_{y  y'} (\Phi_y^{-1} (p) + H^{c}_{x y}
    (v))\\
    & = & \Phi_{y'} (\Phi_{y'}^{-1} \circ h^{c}_{y y'} (p) + H^{c}_{y
    y'} \circ H^{c}_{x
    y} (v))\\
    & = &\Phi_{y'} (\Phi_{y'}^{-1} (p') + H^{c}_{x
    y'} (v))= T_v (p') .
  \end{eqnarray*}
 Consequently, $T_v|_{W^c (p)} = h^u_{p T_v(p)}$ is holomorphic. 
 
 In summary, for any $v\in E^u_x$, $T_v$ is holomorphic along both $W^c$ and $W^u$, therefore by Proposition\ref{cjou},  $T_v$ is a holomorphic diffeomorphism of $W^{cu}(x)$.
\end{proof}

For any $\mathbb{R}$-linearly independent $v_1, v_2 \in E^u (x)$, let $\Lambda_x=\mathrm{span}_ {\mathbb Z} \{v_1,v_2\} $ be a lattice in $E^u_x$. Then $\{T_{v} \mid v \in \Lambda_x  \}$ gives a properly discontinuous holomorphic action of $\Lambda_x$ on $W^{c u} (x)$. Let $N \simeq W^{c u} (x) / \Lambda_x$ be the quotient manifold. 
\begin{lemma}\label{fibra}
    There is a holomorphic submersion $\pi_N:N \to W^c(x)$ such that $\pi_N^{-1}(y)$ is a complex torus of dimension 1 for any $y\in W^c(x)$.
\end{lemma}
\begin{proof}
Recall that $W^u$ holomorphically subfoliates $W^{c u}$ (see Proposition \ref{sho}), there is a natural holomorphic submersion $\pi_W : W^{c u} (x) \rightarrow W^c (x)$. The fiberwise action of $T_v$ along $W^u$ is translation (under the coordinate given by $\Phi$) and commute with $\pi_W$ in the sense:
$$\pi_W \circ T_v =  \pi_W \text{ for any } v \in \Lambda_x. $$
Therefore $\pi_W$ induces a holomorphic submersion $\pi_N : N \rightarrow W^c (x)$. By the definition of $T_v$, for any $y \in W^c(x)$, 
\[ \pi^{- 1}_N (y) = \Phi_y(E^u(y))/ \Phi_y(H^c_{xy}(\Lambda_x)) \simeq E^u(y)/ H^c_{xy}(\Lambda_x)\]
is a complex torus of dimension $1$. 
\end{proof}

We need the following fact from complex geometry:

\begin{lemma}[Section 3 in \cite{cat04}]\label{cgl}
  Let $\pi : N \rightarrow C$ be a holomorphic submersion between compact
  complex manifolds such that all fibers $\pi^{- 1} (y), y \in C$ are complex
  tori of dimension $1$. Then $\pi$ is a holomorphic fiber bundle. In
  particular, these fibers are holomorphic equivalent.
\end{lemma}
We also need the following fact from complex analysis, we denote by $\mathrm{GL}(2,\Lambda):=\{g\in \mathrm{GL}(2,\mathbb R): g\Lambda=\Lambda\}$. 
\begin{lemma}[\cite{ahl06}]\label{cal}
    Let $E_1= \mathbb C / \Lambda$ be a complex torus of dimension $1$. Then for any $A \in \mathrm{GL}(2,\mathbb R)$,  $E_2:=\mathbb C/  A (\Lambda) $ is holomorphic equivalent to $E_1$ if and only if $A \in  \mathbb C^\times \cdot \mathrm{GL}(2, \Lambda)$. 
\end{lemma}

\begin{proof*}{Proof of Proposition \ref{d1}}
  Since $W^u$ holomorphically subfoliate $W^{cu}$, $E^u|_{W^c(x)}$ is a holomorphic line bundle over $W^c(x)$. We claim that $h^c$ is holomorphic in every connected local charts of $E^u|_{W^c(x)}$. For any $y \in W^c(x)$, take $(U, \varphi)$ be a connected local trivialization chart of $y$, i.e., $y\in U\subset W^c(x)$ and %is an open set containing $y$  and there is a holomorphic diffeomorphism 
  $\varphi : E^u|_U \to U \times \mathbb C$ is a holomorphic diffeomorphism. For any $y' \in U$, we have $$\pi_N^{-1}(y') \simeq E^u(y')/  H^c_{xy'}(\Lambda_x) \simeq \mathbb C/ \varphi (H^c_{xy'}(\Lambda_x))$$
By Lemma \ref{fibra} and Lemma \ref{cgl}, $\pi_N^{-1}(y) \simeq \pi_N^{-1}(y')$ are holomorphic equivalent tori, therefore, 
$$\mathbb C/ \varphi (H^c_{xy}(\Lambda_x)) \simeq \mathbb C/ \varphi (H^c_{xy'}(\Lambda_x)) =\mathbb C/  (\varphi \circ H^c_{y y'} \circ \varphi^{-1}) (\varphi (H^c_{xy}(\Lambda_x))). $$
Since $H^c_{yy'}$ is $\mathbb R$-linear, $A_{yy'}:=\varphi \circ H^c_{y y'} \circ \varphi^{-1} \in \mathrm{GL}(2,\mathbb R).$ By Lemma \ref{cal}, $A_{yy'} \in \mathbb C^\times \cdot \mathrm{GL}(2, \varphi (H^c_{xy}(\Lambda_x)))$ for any $y' \in U$. Since $\mathrm{GL}(2, \varphi (H^c_{xy}(\Lambda_x)))$ is discrete, then by continuity of $H^c_{(\cdot ,\cdot)}$ and connectedness of $U$, %$A_{yy'}$ is homotopic to $A_{y y}=\mathrm{Id}$, 
we have $A_{y y'} \in \mathbb C^\times$ for any $y' \in U$. In particular,
$$\Phi_{y'}^{-1} \circ h^{c}_{y y'} \circ \Phi_y=H^{c}_{y y'} =\varphi^{-1} \circ A_{y y'} \circ \varphi $$
is $\mathbb C$-linear and $h^c_{yy'}$ is holomorphic for any $y' \in U$. This implies that $h^c$ is holomorphic on $W^c(x)$ since $W^c(x)$ is covered by finite local trivialization charts. 
\end{proof*}

\section{Proof of Theorem \ref{t2}} \label{sec t2}
We first prove that the center holonomy is holomorphic everywhere.  
\begin{proposition}\label{prop center holo}
    For any $x_1\in M$ and $x_2\in W^c(x_1)$, the center holonomy map $h^{c}_{x_1x_2}$ is holomorphic.
\end{proposition}
\color{red}Before proceeding to the proof, we state a general lemma for partially hyperbolic diffeomorphisms. Let $m$ denote a Lebesgue measure on $M$, and let $m^\ast$ ($\ast = s,c,u$) denote the leafwise Lebesgue measures along $\ast$-leaves. 

\begin{lemma}\label{8.2}
    Let $f:M \to M$ be a $C^2$ partially hyperbolic diffeomorphism and $S\subset M$ be an $f$-invariant open set. If $\mathrm{supp}\, \mu \subset S$ for every ergodic Gibbs $u$-state $\mu$, then ${S} \cap {W}^u_{\mathrm{loc}}({x})$ has full ${m}^u_{{x}}$-measure for every $x \in M$. 
\end{lemma}

\begin{proof}
Assume by contradiction that ${S} \cap {W}^u_{\mathrm{loc}}({x})$ does not have full measure for some $x \in M$. Then the set ${D} := (M \setminus {S}) \cap {W}^u_{\mathrm{loc}}({x})$ has positive $m^u_{{x}}$-measure. By Lemma 11.12 in \cite{bdv04}, any accumulation point $\nu$ of the probability measures
\[ \nu_n = \frac{1}{n} \sum_{j=0}^{n-1} f^j_* \left( \frac{m_{D}}{m^u_x(D)} \right) \]
is a Gibbs $u$-state, where $m_{D}$ is the Lebesgue measure on $D$. Note that $\mathrm{supp}\, \nu_n \subset \bigcup_{i\leq n} f^n(D)$ for all $n \in \mathbb{Z}^+$ and ${M} \setminus {S}$ is an ${f}$-invariant closed set, we obtain $\mathrm{supp}\, {\nu} \subset {M} \setminus {S}$. Since ergodic components of Gibbs $u$-states remain Gibbs $u$-states (see Proposition \ref{gbu}), this contradicts the assumption that $\mathrm{supp}\, {\mu} \subset {S}$ for all ergodic Gibbs $u$-states.
\end{proof}

Now we return to the proof of Proposition \ref{prop center holo}. For any $x\in M$ and $\delta>0$, denote by $W^u_{\delta}(x)$ the disk centered at $x$ with radius $\delta$. Recall that $\pi: M\to \hat M$ is the canonical projection defined in the beginning of Section \ref{sec dic}. Define the set
\begin{equation}\label{eqn: S def} S:= \{{x} \in {M} \mid \exists \delta=\delta(\hat x)>0 \text{ such that } h^c_{x_1x_2}: W^u_{\delta}(x_1) \to W^u(x_2) \text{ is holomorphic}  ~ \forall x_1,x_2 \in W^c({x})\}.\end{equation}
By construction, ${S}$ is an ${f}$-invariant $c$-saturated set. 
\begin{lemma}For any ergodic Gibbs $u$-state $\mu$, we have 
$\mathrm{supp}\,{\mu} \subset S$. 
\end{lemma}

\begin{proof}
Recall that for any ergodic Gibbs $u$-state $\mu$, we have:
\begin{itemize}
    \item $f$ is either a $\mu$-center isometry or a $\mu$-center contraction (Proposition \ref{dic}).
    \item If $f$ is a $\mu$-center isometry, then $h^{c}_{x_1x_2}: W^u_{\mathrm{loc}}(x_1) \rightarrow W^u_{\mathrm{loc}}(x_2)$ is holomorphic for any $\hat x \in \text{supp}\, \hat \mu$ and any $x_1,x_2 \in W^c(\hat{x})$ (Proposition \ref{d1}).
    \item If $f$ is a $\mu$-center contraction, then $h^{c}_{x_1x_2}: W^u_{\mathrm{loc}}(x_1) \rightarrow W^u_{\mathrm{loc}}(x_2)$ is holomorphic for any $\hat x \in \text{supp}\, \hat \mu$ and any $x_1,x_2 \in W^c(\hat{x})$ (Proposition \ref{d2}).
\end{itemize}
Since $W^c$ is a compact foliation, for any $\hat x \in \text{supp}\, \hat \mu$, there exists a uniform constant $\delta = \delta(\hat x)>0$ such that $h^{c}_{x_1x_2}$ is well-defined on $W^u_\delta(x_1)$ for any $x_1,x_2 \in W^c(\hat x)$. Moreover, since $\mathrm{supp}\, \mu$ is always $u$-saturated, the above facts imply that $\mathrm{supp}\,{\mu} \subset S$ for any ergodic Gibbs $u$-state $\mu$. 
\end{proof}

\begin{lemma}
    $S$ is open. 
\end{lemma}

\begin{proof}
Since $S$ is saturated by center leaves, it suffices to show that its projection $\hat S \subset \hat M$ is open. Since $f$ is a fibered partially hyperbolic diffeomorphism, the quotient dynamics $\hat{f}: \hat{M} \to \hat{M}$ is an Anosov homeomorphism with local product structure \cite{gog11}. To prove that $\hat S$ is open, it suffices to show that $\hat W^u_\mathrm{loc}(\hat x) \subset \hat S$ and $\hat W^s_\mathrm{loc}(\hat x) \subset \hat S$ for any $\hat x \in \hat S$. 

It follows immediately from the definition that $\hat W^u_\mathrm{loc}(\hat x) \subset \hat S$. Fix a $\hat x\in \hat S$, we now prove that $\hat y \in \hat S$ for any $\hat y \in \hat W^s_\mathrm{loc}(\hat x)$. Since $\hat x \in \hat S$, there exists $\delta = \delta(\hat x) > 0$ such that $h^c$ is holomorphic on the piece $$L^{cu}_\delta=L^{cu}_\delta(\hat x) := \bigcup_{x \in W^c(\hat x)} W^u_\delta(x)\subset M,$$
Now we define 
$$\mathcal L^{cu}_\delta(\hat y) = h^s(L^{cu}_\delta)\subset M.$$
Note that $\mathcal L^{cu}_\delta(\hat y)$ may not be the same as $\bigcup_{y \in W^c(\hat y)} W^u_\delta(y)$. But there exists a uniform $\delta'>0$ such that 
\begin{equation}\label{eqn: unif low bd}
\mathcal L^{cu}_\delta(\hat y)\supset \bigcup_{y \in W^c(\hat y)} W^u_{\delta'}(y)    
\end{equation}
for any $\hat y \in \hat W^s_\mathrm{loc}(\hat x)$.

The proof that $\hat y \in \hat S$ reduces to showing holomorphicity of $h^c_{y_1y_2}: W^u(y_1) \cap \mathcal L^{cu}_\delta(\hat y) \to W^u(y_2)$ for any $y_1, y_2 \in W^c(\hat y)$, following an argument similar to the proof of Proposition \ref{d2}: for any $w_1 \in  W^u(y_1) \cap\mathcal L^{cu}_\delta(\hat y)$, $w_2 = h^c_{y_1 y_2}(w_1)$, let $z_1= W^s_\mathrm{loc}(w_1) \cap L^{cu}_\delta(\hat x)$ and $z_2= W^s_\mathrm{loc}(w_2) \cap L^{cu}_\delta(\hat x)$. Consider the sequence $\varphi_n = h^c_{f^nz_1 f^nz_2}: W^u_\mathrm{loc}(f^nz_1) \to W^u_\mathrm{loc}(f^nz_2)$. Each $\varphi_n$ is holomorphic since $z_1,z_2 \in  L^{cu}_\delta(\hat x)$.  And we have  $\lim_{n\to+\infty} d_{C^0}(h^c_{f^ny_1 f^ny_2}, \varphi_n) = 0$ since $d(f^n z_1,f^ny_1) \to 0$ and $d(f^n z_2,f^ny_2) \to0$. 

By Lemma \ref{l12 abst}, if $h^c_{y_1y_2}$ is differentiable at $w_1$, then $\bar{\partial}(\Phi_{w_2}^{-1}\circ h^c_{w_1w_2}\circ \Phi_{w_1})(0) = 0$ (here $h^c_{y_1y_2}$ and $h^c_{w_1w_2}$ actually denote the same local map). Similar to the equality in the proof of Lemma \ref{4l2}, we have
\begin{eqnarray*}
\bar{\partial} (\Phi_{y_2}^{-1}\circ h^c_{y_1y_2}\circ \Phi_{y_1})(\Phi_{y_1}^{-1}(w_1))& = & \bar{\partial} (\Phi_{y_2}^{-1} \circ \Phi_{w_2}\circ \Phi_{w_2}^{-1} \circ h^c_{w_1w_2}\circ \Phi_{w_1})(0)\\
(\text{by Lemma }\ref{l9})& = & \bar{\partial} (H^u_{w_2 y_2} \circ (\Phi_{w_2}^{-1} \circ h^c_{w_1w_2}\circ \Phi_{w_1}))(0) \\
& = & 0.
\end{eqnarray*}
Since $h^c_{y_1y_2}$ is quasiconformal (Proposition \ref{cac}), it is differentiable for $m^u_{y_1}$-a.e. $w_1 \in W^u(y_1) \cap \mathcal{L}^{cu}_\delta(y)$. By Lemma \ref{qsl} and the equality of the derivatives above, this implies $\Phi_{y_2}^{-1}\circ h^c_{y_1y_2}\circ \Phi_{y_1}$ is holomorphic and $h^c_{y_1y_2}$ is holomorphic on $W^u(y_1) \cap \mathcal{L}^{cu}_\delta(y)$. Therefore by \eqref{eqn: unif low bd} and the definition of $S$, $\hat y\in \hat S$ for any $\hat y \in \hat W^s_\mathrm{loc}(\hat x)$. Hence we complete the proof of openness of $\hat S$ and $S$.
\end{proof}

\begin{proof*}{Proof of Proposition \ref{prop center holo}}
For any $x_1 \in M, x_2 \in W^c(x_1)$, it follows from the definition of $S$ that $h^c_{x_1x_2}:W^u_\mathrm{loc}(x_1) \to W^u_\mathrm{loc}(x_2)$ is holomorphic on ${S} \cap {W}^u_{\mathrm{loc}}({x_1})$. Applying Lemma \ref{8.2}, we see that ${S} \cap {W}^u_{\mathrm{loc}}({x_1})$ has full ${m}^u_{{x}}$-measure. By Proposition \ref{qsl} and the quasiconformality of center holonomy (Proposition \ref{cac}), we conclude that $h^c_{x_1x_2}$ is holomorphic on the entire $W^u_{\mathrm{loc}}(x_1)$. 
\end{proof*}

\color{black}We then prove that $W^{cs}$ is a holomorphic foliation. Recall that $H^s$ is the stable holonomy of the cocycle $D f|_{E^u}$ (see Proposition \ref{cocycle}). Given $x\in M, y \in W^{cs}_{\mathrm{loc}}(x)$, we let $z$ be the unique intersection point of $W^c_{\mathrm{loc}}(x)$ with $W^s_{\mathrm{loc}}(y)$ and define
$$H^{cs}_{xy}=H^s_{zy} \circ Dh^{c}_{xz}.$$
The following lemma is proved in \cite{bx18} for uniformly $u$-quasiconformal partially hyperbolic diffeomorphism and applies in our setting: 
\begin{lemma}[Lemma 28 in \cite{bx18}]\label{l28}
    Let $f$ be a fibered holomorphic partially hyperbolic diffeomorphism on a compact complex 3-fold. If $h^c$ is $C^1$, then for any $x \in M, y\in W^{cs}(x)$, the center-stable holonomy $h^{cs}_{xy}:W^u_{\mathrm{loc}}(x) \to W^u_{\mathrm{loc}}(y)$ is $C^1$ with derivative $H^{cs}$.
\end{lemma}
Moreover, $H^{cs}$ is $\mathbb C$-linear since both $Dh^c$ and $H^s$ are $\mathbb C$-linear, and we have: 
\begin{proposition}
    $W^{cs}$ is a holomorphic foliation. 
\end{proposition}
\begin{proof}
  Recall that the leaves of $W^{cs}$ are complex manifolds (Lemma \ref{leaf}). By Lemma \ref{l28}, $D h^{cs}=H^{cs}$ is $\mathbb C$-linear and $h^{cs}$ is holomorphic. By Proposition \ref{hf}, $W^{cs}$ is a holomorphic foliation. 
\end{proof}
By applying all arguments to $f^{-1}$, we are able to deduce that $W^{cu}$ is also a holomorphic foliation. As a Corollary: 

\begin{corollary}
    $W^c$ is a holomorphic foliation. In particular, $f$ is a holomorphic skew product over a linear automorphism on a complex $2$-torus.
\end{corollary}

\begin{proof}
Since $W^{cs}$ and $W^{cu}$ are holomorphic foliations, $E^{cs}$ and $E^{cu}$ are holomorphic distributions (See Section \ref{journee}). Their intersection $E^c$ is also a holomorphic distribution, which implies that $W^c$ is a holomorphic foliation. Now, $f$ is a holomorphic skew product over a holomorphic Anosov system $(\hat M, \hat f)$. It follows from Theorem A in \cite{ghy95} that $\hat M$ is a complex torus and $\hat f$ is linear. 
\end{proof}

\section{Proof of Theorem \ref{thm: constru 5d}}\label{sec non holo ex}

The arguments in Section \ref{sec d1} essentially rely on the one-dimensional
assumption of $E^u$. In this section, we provide examples in which $\dim_\mathbb C E^s \geqslant 2$ and $\dim_\mathbb C E^u \geqslant 2$, the center foliation of a holomorphic partially hyperbolic diffeomorphism may fail to be holomorphic.

Let $X= C \times \mathbb T^{4n}$ be a product of a (real) closed surface $C$ and a torus $\mathbb T^{4n}$ of real dimension $4n$. For any linear automorphism $A\in \mathrm{Aut}(\mathbb T^n)$, let
$$f=f_A:=\mathrm{Id}_C\times A \times A \times A \times A:X \to X,$$
then $f$ is a fibered partially hyperbolic diffeomorphism of $X$.
\begin{proposition}
The manifold $X$ admits a complex structure satisfying:
  \begin{itemize}
    \item $X$ is not a holomorphic product of a Riemann surface and a complex torus;
    \item  $f$ is holomorphic, yet the center foliation $W^c$ fails to be holomorphic on any non-empty open subset.
  \end{itemize} 
\end{proposition}

\subsection{Blanchard-Calabi manifolds and fiberwise expanding maps}

Our construction of this complex structure is partially inspired by the construction of Blanchard-Calabi manifolds given by Catanese (Theorem 5.2 in \cite{cat02}), which admit holomorphic fiberwise expanding maps:

    By the classical theory of compact Riemann surfaces, there exists a (very ample) line bundle $L$ over $C$ such that there are two global sections $s_1, s_2 \in H^0 (L)$ without common zeros (see for example Chapter 2, Section 1, Page 215 in \cite{gh14}). Thus, $s \assign (s_1, s_2)$ is a nowhere-vanishing section of the vector bundle $W_0 : = L \oplus L$. 

Motivated by the canonical representation of the quaternion field in the complex matrix group, we define four sections:
$$\sigma_1=s=(s_1,s_2),~\sigma_2=(i s_1,-is_2),~\sigma_3=(-s_2,s_1),~\sigma_4=(is_2,is_1).$$ A straightforward calculation shows that these are four holomorphic sections of $W_0$ that are everywhere $\mathbb{R}$-linearly independent. Let $X_0$ be the quotient of $W_0$ by the fiberwise $\mathbb{Z}^4$-action given by translations $v \mapsto v + \sum_{k=1}^4 a_k \sigma_k$. As a real manifold, $X_0$ is diffeomorphic to $C \times \mathbb{T}^4$, but as a complex manifold it is not biholomorphic to any product of $C$ with a complex torus (using a similar argument as in the proof of Proposition \ref{prop non-holo}).

A crucial observation is that $X_0$ admits a holomorphic endomorphism that is fiberwise expanding. Under a local coordinate chart $U_1$, the point in the total space of $W$ can be written as $(z_0, \mathbf{z_1})$ where $z_0$ denotes the coordinate along $C$, and $\bf{z_1}$ denotes the coordinate along the fibers of $W_0$ ($\bf{z_1}$ is an vector in $\mathbb C^2$). Then $\phi: (z_0,\mathbf{z_1}) \mapsto (z_0,a \mathbf{z_1})$, where $a \in \mathbb{Z}$ is a locally well-defined holomorphic fiberwise expanding map for any $a>1$. Now, for any $U_2$ such that $U_1 \cap U_2 \neq \emptyset$, let $g_{12}:U_1 \cap U_2 \to GL(2,\mathbb{C})$ be the transition function. Then $g_{12}^{-1} \circ\phi \circ g_{12}$ is still of the form $(z_0,\mathbf{z_1}) \mapsto (z_0,a \mathbf{z_1})$. Therefore, $\phi$ is independent of the local trivialization and can be extended to a holomorphic fiberwise expanding map defined on the ambient $X_0$. Note that $\phi$ preserves the lattice given by $\sigma_k$, and thus $\phi$ defines a holomorphic fiberwise expanding map on $W_0$.

\subsection{Holomorphic partially hyperbolic systems on Blanchard-Calabi manifolds}
In the real setting, $S^1$ admits expanding maps. By taking $n$-copies of $S^1$, we are able to construct Anosov diffeomorphisms on $\mathbb{T}^n$. Similarly, by taking several copies of $W_0$, we are able to construct a holomorphic partially hyperbolic diffeomorphism. 

Let the bundle $W = \oplus_{l=1}^n W_0 $ be the $n$-copies of $W_0$, and let \red{$\sigma_{1,l} =
(0,\ldots,\sigma_1,\ldots,0)$, $\sigma_{2,l} = (0,\ldots,\sigma_2,\ldots,0)$, $\sigma_{3,l} = (0,\ldots,\sigma_3,\ldots,0)$, and $\sigma_{4,l} = (0,\ldots,\sigma_4,\ldots,0)$}, where the non-zero component is in the $l$-th component. Then the quotient $X$ of $W$ by the $\mathbb{Z}^{4 n}$-action acting
fiberwise by translations: $v \mapsto v + \sum_{k, l} a_{k, l}
\sigma_{k, l}$ is a complex manifold $C^\omega$-diffeomorphic to the real analytic manifold $C \times \mathbb T^{4n}$.

We then define the holomorphic partially hyperbolic diffeomorphism $f$. Under a local coordinate chart, the point in the total space of $W$ can be written as $(z_0, \bf z_1, \ldots, z_n)$ where $z_0$ denotes the coordinate along $C$, and $\bf z_l$ $(l=1, \ldots, n)$ denotes the coordinate of the $l$-th $W_0$-component in $W$. For any given hyperbolic matrix $A \in \mathrm{SL}(n,\mathbb Z)$, i.e., the eigenvalues of $A$ do not lie on the unit circle, let
\[ f_W :W \to W, \quad (z_0, \mathbf{z_1, \ldots, z_n}) \mapsto (z_0, A\cdot (\bf z_1, \ldots, z_n)^{\Tau}). \]
As mentioned in the previous example, the fiberwise expanding is well-defined and holomorphic. Taking into account that the map $t_{ij}$ which add the $i$-th coordinate to the $j$-th component is also holomorphic, $f$ is then a well-defined holomorphic diffeomorphism of $W$. Note that $f_W$ preserves the lattice generated by $\sigma_{k l}$, and $f_W$ induces a
holomorphic fibered partially hyperbolic diffeomorphism $f$ on $X$. 

$X$ is a holomorphic fibration over $C$ with fibers being complex tori. \red{By analyzing the subbundle structure of $W^\vee$ through its dual exact sequence and the Betti number constraints on Kähler manifolds, one can show that such $X$ is generally non-Kähler:}

\begin{proposition}[Remark 7.3 in \cite{cat04}]\label{cat kahler}
    The following statements are equivalent:
    \begin{itemize}
        \item $X$ is Kähler;
        \item $X = C \times T$ is a holomorphic product;
        \item $W$ is trivial. 
    \end{itemize}
\end{proposition}

From the viewpoint of real dynamics, $(f,X)$ is quite simple. The underlying real manifold of $X$ is a product manifold $X_\mathbb{R} = C \times \mathbb{T}^{4n}$, $f= \mathrm{Id}_C \times A \times A \times A \times A$ is a product map as a real diffeomorphism. The center holonomy $h^c$ is just the identification map between these real tori, which is real analytic. However, $(f,X)$ is more complicated as a holomorphic dynamic system. As we shall demonstrate in the following proposition, the fibers ($su$-leaves) are not holomorphically equivalent. Consequently, the center holonomy is not holomorphic, and $W^c$ is not a holomorphic foliation.

\begin{proposition}\label{prop non-holo}
 $X$ is non-Kähler and $W^c$ is not a holomorphic foliation. Moreover, $W^c$ is not holomorphic on any open set. 
\end{proposition}

\begin{proof}
By Proposition \ref{cat kahler}, $X$ is Kähler if and only if the bundle $W$ is trivial. It follows directly from our construction of $W$ that $W$ is non-trivial, and therefore, $X$ is not a holomorphic product. To be precise, since $L$ is a very ample line bundle, we have $\deg L>0$, and the determinant bundle of $W=\oplus_{l=1}^{2n} L$ has degree
$$\deg (\det W)= n \deg L > 0.$$
Consequently, $W$ is nontrivial and $X$ is non-Kähler. 

%(The degree of a line bundle is a numerical invariant associated with it, which is often used in algebraic geometry to measure the ``twisting" of a line bundle over a variety or a manifold. We refer to \cite{gh14} for more details.)

Assume that there is some $x \in M$ and a leafwise neighborhood $x \in U \subset W^c(x)$ such that $E^c$ is holomorphic on a neighborhood of $U$, then $h^c_{xy}|_{W^{su}_\mathrm{loc}(x)}: W^{su}_\mathrm{loc}(x) \to W^{su}_\mathrm{loc}(y)$ is holomorphic for any $y \in U$. Notice that $h^c_{xy}$ is a $\mathbb R$-linear diffeomorphism between complex tori $W^{su}(x), W^{su}(y)$.  
Therefore $h^c_{xy}:W^{su}(x) \to W^{su}(y)$ is holomorphic on the ambient $W^{su}(x)$ since a $\mathbb R$-linear diffeomorphism between complex tori is holomorphic if its derivative is $\mathbb C$-linear at a point. In particular, $W^{su}(x)$ and $W^{su}(y)$ are holomorphically equivalent and $X|_{W^{su}(U)}$ is holomorphically equivalent to $U \times W^{su}(x)$.

Let $X' = C \times W^{su}(x)$ be a holomorphic product; the underlying real manifold $X'_\mathbb R$ is also $C \times \mathbb T^{4n}$. We denote $J': T X'_\mathbb R \to TX_\mathbb R$ to be the complex structure on $X'$. Let $\tau: X_\mathbb R \to X'_\mathbb R$ be the identity map between the underlying real manifold of $X$ and $X'$, and $\tau$ is real analytic. Since $J$ and $J'$ are real analytic tensor fields, $\frac{\partial \tau}{\partial \bar z_i}$ is a real analytic function on $X$ for any $i=1,2,\ldots,2n+1$. By our construction, $\tau:X \to X'$ is holomorphic on the open set $U \times W^{su}(x)$, i.e., $\frac{\partial \tau}{\partial \bar z_i}=0$ on $U \times W^{su}(x)$. By the identity theorem for real analytic functions (a real analytic function on a connected domain vanishes if and only if it vanishes on a non-empty open subset of this domain), we see that $\frac{\partial \tau}{\partial \bar z_i}=0$ and $\tau$ is holomorphic on the ambient manifold $X$. In particular, $X$ is holomorphically equivalent to $X'$, which is a contradiction. 
\end{proof}
\begin{proof}[Proof of Theorem \ref{thm: constru 5d}] The proof above provide odd-dimensional examples ($n\geq 5$) that satisfy Theorem \ref{thm: constru 5d}. By taking product with translations on elliptic curves we get even-dimensional examples that satisfy Theorem \ref{thm: constru 5d}. \end{proof}

\section{Several examples}\label{sec exa}
In this section, we give examples to illustrate the accessibility classes for holomorphic fibered partially hyperbolic diffeomorphisms could be complicated. 

We say that a point $y \in M$ is $su$-accessible from $x \in M$ if there exists an $su$-path, i.e., a path $\gamma : [0, 1] \rightarrow M$ piecewise contained in the leaves of the strong stable and strong unstable foliations, joining $x$and $y$. This defines an equivalence relation on $M$. We denote by $$\mathrm{Acc} (x)
= \{ y \in M : y \text{ is  $su$-accessible from } x \}$$ the \textit{accessibility class} of $x$. If $\mathrm{Acc}(x) = M$, then $f$ is called \textit{accessible}. 
\begin{example}[Accessible case: Iwasawa manifolds]
    Let $\mathfrak{g}_3(\mathbb{C})$ be the Heisenberg algebra of dimension 3 over $\mathbb{C}$, i.e., the non-trivial Lie brackets are given by $[e_1, e_2] = e_3$ and $\mathfrak{g}_3(\mathbb{C})$ admits a complex partially hyperbolic automorphism
    \[ f : e_1 \mapsto 2 e_1 + e_2, \quad e_2 \mapsto e_1 + e_2, \quad e_3 \mapsto e_3. \]
    Moreover, $f$ preserves the lattice $\gamma = \mathrm{span}_\mathbb{Z}\{e_1,...,e_6\} \subset \mathfrak{g}_3(\mathbb{C})$. Let $F$ be the lift of $f$ on the complex Heisenberg group $H_3(\mathbb{C}) = \mathrm{Lie}(\mathfrak{g}_3(\mathbb{C}))$ via the exponential map. $F$ is then a holomorphic partially hyperbolic automorphism of $H_3(\mathbb{C})$ preserving the lattice $\Gamma := \exp(\gamma)$. $F$ also induces a holomorphic partially hyperbolic diffeomorphism on the complex Heisenberg manifold (Iwasawa manifold) $H_3(\mathbb{C}) / \Gamma$. It follows from the construction that $[E^{s}, E^{u}] = E^c$, and $F$ is accessible. 
\end{example}

Historically, a complex nilmanifold meant a quotient of a complex nilpotent Lie group over a cocompact lattice. An example of such a nilmanifold is an Iwasawa manifold as in the previous example. From the 1980s, another (more general) notion of a complex nilmanifold replaced this one. 

An almost complex structure on a real Lie algebra is an endomorphism $J:\mathfrak{g} \to \mathfrak{g}$ which squares to $-\mathrm{Id}_{\mathfrak{g}}$.  $J$ defines a left-invariant almost complex structure on the corresponding Lie group $G$ by multiplication. Such a manifold $(G,J)$ is a complex manifold if the Nijenhuis tensor vanishes (see Proposition \ref{nn}). A complex nilmanifold is a quotient of a complex group manifold $(G,J)$, by a discrete cocompact lattice, acting from the right.

This construction indeed provides more interesting examples. We will construct holomorphic partially hyperbolic systems on such complex nilmanifold, where the accessibility classes are not complex submanifolds:

\begin{example}[Proof of Theorem \ref{thm: 5d acc class}]\label{5dacc}
Let $\mathfrak{g} = \mathfrak{g}_5(\mathbb{R}) \oplus I$ be the direct sum of the real Heisenberg algebra of dimension 5 and a 1-dimensional abelian ideal, i.e., the non-trivial Lie brackets are given by $[e_1, e_2] =[e_3, e_4] = e_5$. $\mathfrak{g}$ admits a partially hyperbolic automorphism
\[ f : e_1 \mapsto 2 e_1 + e_3, \quad e_2 \mapsto 2e_2 + e_4, \quad e_3 \mapsto e_1 + e_3, \quad e_4 \mapsto e_2 + e_4, \quad e_5 \rightarrow e_5, \quad e_6 \mapsto e_6. \]
$f$ can be lifted to a partially hyperbolic automorphism $F$ on $G = H_5(\mathbb{R}) \times S^1$ via the exponential map. Since $f$ preserves the lattice $\gamma = \mathrm{span}_\mathbb{Z}\{e_1,...,e_6\} \subset \mathfrak{g}$, $F$ induces a partially hyperbolic diffeomorphism on $G/ \Gamma$, where $\Gamma = \exp(\gamma)$. In particular, $F$ is the identity along $S^1$, and the accessibility classes are diffeomorphic to $H_5(\mathbb{R})$.

We then endow $G/ \Gamma$ with a complex structure. Let $J: \mathfrak{g} \to \mathfrak{g}$ be an almost complex structure given by
\[ J e_1 = e_2, \quad J e_3 = e_4, \quad J e_5 = e_6. \]
It is straightforward to check that the Nijenhuis tensor $N_{J}$ vanishes. Therefore $J$ gives left-invariant complex structures on $G$ and $G/ \Gamma$. It is also a direct check that $f$ commutes with $J$. Consequently, the lift $F$ gives holomorphic partially hyperbolic diffeomorphisms on $(G, J)$ and $(G/ \Gamma,J)$. As mentioned in the previous paragraph, the accessibility classes of $F$ are of real dimension 5 and cannot be a complex submanifold.
\end{example}

We end this section with a holomorphic partially hyperbolic system with a compact center that is not fibered. 
\begin{example}
    Let $E = \mathbb{C}/ \langle 1, \tau \rangle$ be an elliptic curve and $M = E \times E \times E$. Then $A = \begin{pmatrix}2&1&0\\1&1&0\\0&0&1\end{pmatrix} \in \mathrm{GL}(3,\mathbb{C})$ is a holomorphic partially hyperbolic diffeomorphism on $M$. Let $\sigma : M \to M, (z_1,z_2,z_3) \to (-z_1, -z_2, z_3 + \frac{1}{2})$ be a holomorphic periodic automorphism of $M$. It is straightforward to check that $\sigma \circ A = A \circ \sigma$. Therefore, $A$ induces a holomorphic partially hyperbolic diffeomorphism $A_\sigma$ on $M_\sigma = M/ \sigma$. $(M_\sigma, f_\sigma)$ is then a holomorphic partially hyperbolic diffeomorphism with compact center. All but finite center fibers are biholomorphic to $E$ except for 16 singular fibers biholomorphic to $\mathbb{C}/ \langle \frac{1}{2}, \tau \rangle$. In particular, the quotient space $M_\sigma / W^c$ is an orbifold, and $f_\sigma$ is not fibered. 
\end{example}

\appendix
\section{Proof of Lemma \ref{liag}}\label{sec liag}
Without loss of generality, we assume that $d(\cdot, \cdot)$ is induced by the metric of constant curvature. Let $A_n=U_n D_n V_n$ be the KAK decomposition of $A_n$, where $D_n$ is a diagonal matrix, $U_n$, and $V_n$ are unitary matrices. Then we could take a subsequence $n_k$ such that $U_{n_k}$ and $V_{n_k}$ converge to $U_0$ and $V_0$ respectively. As a consequence, $\lim_{k\to \infty} A_{n_k}/\|A_{n_k}\|=U_0\cdot \begin{pmatrix}
    1&0\\0&0
\end{pmatrix}\cdot V_0:=A_0$. Let $b:=V_0^{-1}\cdot \begin{pmatrix}
    0\\1
\end{pmatrix}$ be the kernel direction of $A_0$. Then for any $x_1, x_2 \in \mathbb P^1 \backslash \{b\}$, we know that for $k$ sufficiently large, $x_1$ and $x_2$ are uniformly bounded away from $V_{n_k}^{-1}\cdot \begin{pmatrix}
    0\\1
\end{pmatrix}$; hence, for $k$ sufficiently large (and therefore $\|A_{n_k}\|$ sufficiently large), $A_{n_k}\cdot x_1$ and $A_{n_k}\cdot x_2$ are sufficiently close to $U_{n_k}\cdot \begin{pmatrix}
  1\\0  
\end{pmatrix}$. Both of them will converge to $U_{0}\cdot \begin{pmatrix}
  1\\0  
\end{pmatrix}$, i.e., the image direction of $A_0$. Therefore,
     \[ \lim_{k \rightarrow \infty} d (A_{n_k} x_1, A_{n_k} x_2) = 0 \]
for any $x_1, x_2 \in \mathbb P^1 \backslash \{b\}$. 

\section{Proof of Proposition \ref{reli} }\label{proof reli}
We consider the space 
\[ 
\mathcal{E} = \{ (x, y) \in M^2 : x \in M, y \in W^c (x) \},
\]
and we define the dynamics \( f_{\mathcal{E}} : \mathcal{E} \to \mathcal{E}, \quad f_{\mathcal{E}}(x, y) = (f(x), f(y)) \).

\begin{lemma}[Proposition 20 in \cite{bx18}]
The space $\mathcal{E}$ is a continuous fiber bundle over $M$ with compact fibers. Moreover $f_{\mathcal{E}}$ preserves an invariant probability measure $\nu_{\mathcal{E}}$ on $\mathcal{E}$ such that for any continuous function $\varphi$,
  $$\int_{\mathcal E}\varphi d\nu_{\mathcal{E}}=\int_M\int_{W^c(x)}\varphi(x,y)d\nu_x^c(y)d\nu(x).$$
\end{lemma}

\begin{lemma}
  \label{dif}
  Let
  \[ Q \assign \{ x \in M : \text{for }\nu_x^c \text{-a.e. } y \in W^c (x), h^{c}_{xy} : W^u_{\mathrm{loc}} (x) \rightarrow W^u_{\mathrm{loc}} (y) \text{ is differentiable at } x \} . \]
  Then $\nu (Q) = 1$.
\end{lemma}

\begin{proof}
  For $x \in M$, let $\nu^u_x$ denote the conditional measure of $\nu$ on
  $W^u_{\mathrm{loc}} (x)$, and let $\nu^{cu}_x$ be the conditional measure of
  $\nu$ on the subset $L_{\mathrm{loc}}^{cu}(x):=W^c(W^u_{\mathrm{loc}} (x))$. Since $\nu$ is a Gibbs $cu$-state, by Fubini's theorem (Proposition \ref{fub}), $\nu^{cu}_x$ decomposes as conditional measures in two different ways:
  \[ \nu^{cu}_x \asymp \int_{W^c (x)} \nu^u_y d \nu^c_x (y) \asymp
     \int_{W^u_{\mathrm{loc}} (x)} \nu^c_y d \nu^u_x (y). \]
 By Proposition \ref{cac}, for every $y \in L_{\mathrm{loc}}^{cu}(x)$, the center holonomy $h^{c}_{yx} : W^u_{\mathrm{loc}} (y) \rightarrow W^u_{\mathrm{loc}} (x)$ is differentiable at $\nu^u_y$-a.e. $z \in W^u_{\mathrm{loc}} (y)$. Thus, if we set
  \[ T_x = \{ z \in L_{\mathrm{loc}}^{cu}(x) : h^{c}_{y x} \text{ is differentiable at } z \text{ for } y = W^c (x) \cap W^u_{\mathrm{loc}} (z) \}, \]
  then by the first expression for $\nu^{cu}_x$, we have $\nu^{cu}_x (T_x) = \nu^{cu}_x (L_{\mathrm{loc}}^{cu}(x))$ for $\nu$-almost $x \in M$. Since $T_x$ has full $\nu^{cu}_x$ measure in $L_{\mathrm{loc}}^{cu}(x)$, we conclude by the second expression for $\nu^{cu}_x$ that $\nu^c_y (T_x \cap W^c (y)) = \nu^c_y (W^c (y))$ for $\nu^u_x$-almost $y \in W^u_{\mathrm{loc}} (x)$. This immediately implies that $\nu^u_x (Q \cap W^u_{\mathrm{loc}}  (x)) = \nu^u_x (W^u_{\mathrm{loc}}  (x))$ from the definition of $Q$. Since $\nu$ is absolutely continuous along $W^u$, and this holds for $\nu$-almost $x \in M$, we conclude that $\nu (Q) = 1$.
\end{proof}

We then define $\mathcal{Q}= \{ (x, y) \in \mathcal{E}: x \in Q \}$. From the definition of $Q$ and Lemma \ref{dif}, we see that $\mathcal{Q}$ has full $\nu_{\mathcal{E}}$-measure inside $\mathcal{E}$. For $(x, y) \in \mathcal{Q}$, we can define $H^{c}_{xy} : E^u_x \rightarrow E^u_y$ to be the derivative of $h^{c}_{xy}$ at $x$. The map $(x, y) \rightarrow H^{c}_{xy}$ is clearly measurable and defined $\nu_{\mathcal{E}}$-a.e. Our next goal is to show that the maps $H^{c}$ are equivariant with respect to the unstable holonomy $H^u$ of $Df|_{E^u}$. Before going into the proof, we need the following measure-theoretic lemma.

\begin{lemma}[Lemma 21 in \cite{bx18}]
  \label{l21}
  Let $T$ be a measure-preserving transformation of a finite measure space $(X, \mu)$, and let $\{ K_n \}_{n \geqslant 1}$ be a sequence of measurable subsets of $X$ with $\sum_{n = 1}^{\infty} \mu (X \backslash K_n) < + \infty$. Then there is a full measure subset $\Omega \subset X$ with the property that if $x, y \in \Omega$, then there is an $n \in \mathbb{N}$ and a sequence $n_k \rightarrow \infty$ with $T^{n_k} (x) \in K_n, T^{n_k} (y) \in K_n$ for each $n_k$.
\end{lemma}

\begin{lemma}
  \label{l24}
  There is a full $\nu_{\mathcal{E}}$-measure subset $\mathcal{Q}'
  \subset \mathcal{Q}$ such that if $(x, y), (z, w) \in \mathcal{Q}'$ with $z \in W^u_{\mathrm{loc}} (x)$ and
  $w = W^u_{\mathrm{loc}} (y)$, then
  \begin{equation}
    H^{c}_{z w} \circ H^u_{x z} = H^u_{y w} \circ H^{c}_{x y}
    \label{5.2}
  \end{equation}
\end{lemma}

\begin{proof}
  By Lusin's theorem, we can find an increasing sequence of compact subsets $K \subset \mathcal{Q}$ such that $\nu_{\mathcal{E}} (\mathcal{E} \backslash K_n) < 2^{-n}$ and such that $H^{c}$ restricts to a uniformly continuous function on each $K_n$. Since $\nu_{\mathcal{E}}$ is $f_{\mathcal{E}}$-invariant, by applying Lemma \ref{l21} to $f_{\mathcal{E}}^{-1}$, there is a measurable set $\mathcal{Q}'\subset \mathcal{Q} $ with $\nu_\mathcal{E} (\mathcal{E} \backslash \mathcal{Q}') = 0$ and such that for any pair of points $(x, y), (z, w) \in \mathcal{Q}'$, there is an $n \in \mathbb{N}$ and a pair of infinite sequences $n_k \rightarrow \infty$ and $n_k' \rightarrow \infty$ with $f_{\mathcal{E}}^{-n'_k} (x, y), f_{\mathcal{E}}^{-n'_k} (z, w) \in K_n$.
  
Let $(x, y), (z, w) \in \mathcal{Q}'$ be such that $z \in W^u_{\mathrm{loc}} (x)$ and $w \in W^u_{\mathrm{loc}} (y)$. Since $H^c$ is uniformly continuous on $K_n$ and $d (f^{-n} x, f^{-n} z), d (f^{-n} y, f^{-n} w) \rightarrow 0$ as $n \rightarrow + \infty$, we conclude that
\[ H^u_{f^{-n_k} y f^{-n_k} w} \circ H^{c}_{f^{-n_k} x f^{-n_k} y} \circ \left( H^{c}_{f^{-n_k} z f^{-n_k} w} \circ H^u_{f^{-n_k} x f^{-n_k} z} \right)^{- 1} \rightarrow \mathrm{Id} \]
as $k \rightarrow + \infty$, where $\{ n_k \}$ is the infinite sequence from the previous paragraph corresponding to the pair $(x, y), (z, w)$.

  From differentiating the equation
\[ h^{c}_{x y} = f^{n_k} \circ h^{c}_{f^{-n_k} x f^{-n_k} y} \circ f^{-n_k}, \]
we obtain the equation
\[ H^{c}_{x y} = D f_{f^{-n_k }y}^{n_k} |_{E^u} \circ H^{c}_{f^{-n_k} x f^{-n_k} y} \circ D f_x^{-n_k} |_{E^u} . \]
By Proposition \ref{cocycle} we have
\[ H^u_{y w} = D f_{f^{-n_k}w}^{n_k} |_{E^u} \circ H^u_{f^{-n_k} y f^{-n_k} w} \circ D f_y^{-n_k} |_{E^u}. \]
Therefore
\[ H^u_{y w} \circ H^{c}_{x y} = D f_{f^{-n_k}w}^{n_k} |_{E^u} \circ H^u_{f^{-n_k} y f^{-n_k} w} \circ H^{c}_{f^{-n_k} x f^{-n_k} y} \circ D f_x^{-n_k} |_{E^u} . \]
Similarly,
\[ H^{c}_{z w} \circ H^u_{x z} = D f_{f^{-n_k}w}^{n_k} |_{E^u} \circ H^{c}_{f^{-n_k} z f^{-n_k} w} \circ H^u_{f^{-n_k} x f^{-n_k} z} \circ D f_x^{-n_k} |_{E^u} . \]
 Note that $D f^{n_k} |_{E^u}$ acts on $E^u$ by multiplication of a complex number. We have
 \begin{eqnarray*}
    &  & (H^u_{y w} \circ H^{c}_{x y}) \circ (H^{c}_{z w} \circ H^u_{x
    z})^{- 1}\\
    & = & D f_{f^{-n_k}w}^{n_k} |_{E^u} \circ \left( H^u_{f^{-n_k} y f^{-n_k} w}
    \circ H^{c}_{f^{-n_k} x f^{-n_k} y} \circ \left( H^{c}_{f^{-n_k} z
    f^{-n_k} w} \circ H^u_{f^{-n_k} x f^{-n_k} z} \right)^{- 1} \right) \circ D f_{w}^{-n_k} |_{E^u} \\
    & \rightarrow & \mathrm{Id}
  \end{eqnarray*}
  as $k \rightarrow + \infty$. Therefore we deduce equation \eqref{5.2} as desired. 
\end{proof}

\begin{lemma}
  \label{l27}For any $x, y \in \mathrm{supp} \nu$, $y \in W^c (x)$, we have the
  equality
  \[ \Phi_y^{-1} \circ h^{c}_{x y} \circ \Phi_x = H^{c}_{x y} \]
  as maps from $E^u_x$ to $E^u_y$. The measurable function $H^{c}$ on
  $\mathcal{Q}'$ therefore admits a continuous extension to $\mathrm{supp} \nu_\mathcal{E}$, and
  the center holonomy is linear in the charts $\{ \Phi_x \}_{x \in M}$.
\end{lemma}

\begin{proof}
 We first consider pairs $(x, y) \in \mathcal{Q}'$. Since $D_0 \Phi_x = 1$
  for every $x \in M$ and $H^{c}_{x y} (x) = y$, the equation
  \[ D_0 (\Phi_y^{- 1} \circ h^{c}_{x y} \circ \Phi_x) = H^{c}_{x y} \]
  holds for any $(x, y) \in \mathcal{Q}'$. To compute the derivative at other
  points of $E^u_x$, we let $v \in E^u_x, z = \Phi_x (v)$, and $w = h^{c}_{x y} (z)$. We suppose that $(z, w) \in
  \mathcal{Q}'$ and compute,
  \[ D_v (\Phi_y^{- 1} \circ h^{c}_{x y} \circ \Phi_x) = D_0 (\Phi_y^{- 1}
     \circ \Phi_w) \circ D_0 (\Phi_w^{- 1} \circ h^{c}_{x y} \circ \Phi_z)
     \circ D_v (\Phi_z^{- 1} \circ \Phi_x) . \]
  By Lemma \ref{l9} we know that $D_0 (\Phi_y^{- 1} \circ \Phi_w) =
  H^u_{w y}$ and $D_v (\Phi_z^{- 1} \circ \Phi_x) = H^u_{x z}$. Hence by Proposition \ref{l24}
  \[ D_v (\Phi_y^{- 1} \circ h^{c}_{x y} \circ \Phi_x) = H^u_{w y} \circ
     H^{c}_{z w} \circ H^u_{x z} = H^{c}_{x y}, \]
  whenever $(x, y), (z, w) \in \mathcal{Q}'$, by Lemma \ref{l24}. Since
  $\mathcal{Q}'$ has full $\nu_{\mathcal{E}}$ measure and $\nu$ is a Gibbs $cu$-state, we conclude that for
  $\nu$-a.e. $x \in M$ and $\nu^c_x$-a.e. $y \in U_x$ the map $\Phi_y^{- 1}
  \circ h^{c}_{x y} \circ \Phi_x : E^u_x \rightarrow E^u_y$ is
  differentiable Lebesgue almost everywhere on $E^u$ with derivative $H^{c}_{x
  y}$ everywhere. By Lemma \ref{qsl} and quasiconformality of center holonomy (Proposition \ref{cac}), this implies that for fixed $x$ and $y$, $\Phi_y^{- 1} \circ
  h^{c}_{x y} \circ \Phi_x$ is a $C^1$ map with derivative $H^{c}_{x y}$
  everywhere, i.e., $\Phi_y^{-1} \circ H^{c}_{x y} \circ \Phi_x$ coincides exactly
  with the linear map $H^{c}_{x y}$. 
  
 Since $\mathcal{Q}'$ has full $\nu_{\mathcal{E}}$-measure in $\mathcal{E}$
  we conclude that $\mathcal{Q}'$ is dense in $\mathrm{supp} \nu_{\mathcal{E}}$
  and thus the equation
  \begin{equation}\label{4.5}
      \Phi_y^{- 1} \circ h^{c}_{x y} \circ \Phi_x = H^{c}_{x y}
  \end{equation}
  holds on $E^u$ for a dense set of pairs $(x, y) \in \mathrm{supp}
  \nu_{\mathcal{E}}$. But the left side of this equation depends uniformly
  continuous on the pair $(x, y)$ as a map $\Phi_x^{- 1} (W^u_{\mathrm{loc}}
  (x)) \rightarrow \Phi_y^{- 1} (W^u_{\mathrm{loc}} (y))$ between neighborhoods
  of $0$ in $E^u_x$ and $E^u_y$ of uniform size. Furthermore, the linear map
  $H^{c}_{x y}$ is determined by its restriction to a map between these
  neighborhoods. Hence we conclude that $H^{c}_{x y}$ also depends
  uniformly continuously on the pair $(x, y) \in \mathcal{Q}'$.

 When $y$ is close to $x$, the map $\Phi_y^{-1} \circ h^{c}_{x y} \circ
  \Phi_x$ is uniformly close to the linear identifications $I_{x y} : E^u_x
  \rightarrow E^u_y$. Hence, $H^{c}_{x y}$ is uniformly close to $I_{x y}$
  for $(x, y) \in \mathcal{Q}'$. In particular, for $(x, y) \in \mathcal{Q}'$,
  the maps $H^{c}_{x y}$ belong to a uniformly bounded subset of the space
  of invertible linear maps $E^u_x \rightarrow E^u_y$. This shows that $H^{c}_{x y}$ admits a continuous extension to $\mathrm{supp} \nu_{\mathcal{E}}$
  such that equation $\left( \ref{4.5} \right)$ still holds on a neighborhood
  of $0$ in $E^u_x$ for any $(x, y) \in \mathrm{supp} \nu_{\mathcal E}$. Finally, because
  $\Phi_x, h^{c}_{x y}$, and $H^{c}_{x y}$ all have the proper equivariance
  properties with respect to $f^{-1}$ and $D f^{-1}$, which uniformly
  contract $E^u$, it follows that equation $\left( \ref{4.5} \right)$ actually
  holds on all of $E^u_x$.
\end{proof}
\bibliographystyle{plain}
\bibliography{main1}

\end{document}